\documentclass[a4paper,10pt]{article}

\usepackage{multido}
\usepackage{amssymb,amsmath,amsthm}

\usepackage{mathrsfs}

\theoremstyle{plain}
\newtheorem{Def}{Definition}[section]
\newtheorem{Theor}[Def]{Theorem}
\newtheorem{Prop}[Def]{Proposition}
\newtheorem{Lem}[Def]{Lemma}

\theoremstyle{remark}
\newtheorem{Rem}[Def]{Remark}

\newcommand{\cR}{{\mathbb R}}
\newcommand{\non}{\nonumber}
\newcommand{\eq}[1]{\mbox{\rm {(\ref{#1})}}}
\newcommand{\ve}{\varepsilon}
\newcommand{\na}{\nabla}

\DeclareMathOperator{\di}{div}
\DeclareMathOperator{\Curl}{Curl}

\DeclareMathOperator{\inter}{int}
\DeclareMathOperator{\rt}{rt}

\newcommand{\yieldlimit}{\sigma_{\mathrm{y}}}

\newcommand{\R}{\mathbb{R}}

\DeclareMathOperator{\sym}{sym}
\DeclareMathOperator{\Tr}{tr}

\DeclareMathOperator{\dev}{dev}

\begin{document}
\title{Homogenization for dislocation based gradient visco-plasticity}
\author{Sergiy Nesenenko%
\thanks{Sergiy Nesenenko, Fakult\"at f\"ur  Mathematik, Universit\"at Duisburg-Essen, Campus Essen, Thea-Leymann Strasse 9, 45117 Essen, Germany, email: sergiy.nesenenko@uni-due.de, Tel.: +49-201-183-2827}
}
\maketitle
\begin{center}
\vspace{-2ex}{\it\large Communicated with Patrizio Neff}
\end{center}

\begin{abstract}
In this work we study the homogenization for infinitesimal dislocation based
gradient viscoplasticity with linear kinematic hardening and general non-associative monotone plastic flows. The constitutive equations in the models we 
study are assumed to be only of monotone type. Based on the generalized version of Korn's inequality for incompatible tensor fields (the non-symmetric plastic distortion)
due to Neff/Pauly/Witsch, we derive uniform estimates for the  solutions of quasistatic initial-boundary
 value problems 
under consideration and then using a modified unfolding operator technique 
and a monotone operator method we obtain the homogenized system of equations.  
A new unfolding result for the $\Curl\Curl$-operator is presented in this work as well. 
The proof of the last result is
based  on the Helmholtz-Weyl decomposition for vector fields in general
$L^q$-spaces. 
\end{abstract}

\\[2ex]
\section{Introduction}

We study the homogenization of
quasistatic initial-boundary value problems arising in gradient viscoplasticity. The models
we study use rate-dependent constitutive equations with internal variables to describe the
deformation behaviour of metals at infinitesimally small strain. 

Our focus is on a phenomenological model on the macroscale not including the case of single crystal plasticity. 
Our model has been first presented in \cite{Neff_Chelminski07_disloc}. It is inspired by the early work of Menzel and Steinmann \cite{Steinmann00}. Contrary to more classical strain gradient approaches, the model features from the outset a non-symmetric plastic distortion field $p\in{\cal M}^3$ \cite{Bardella10},
a dislocation based energy storage based solely on $|\Curl p|$ (and not $\nabla p$) and therefore second gradients of the plastic 
distortion in the form of $\Curl\Curl p$ acting as dislocation based kinematical backstresses. We only consider energetic length scale effects and not higher gradients in the dissipation.

Uniqueness of classical solutions in the subdifferential case (associated plasticity) for rate-independent and rate-dependent formulations is shown in  \cite{Neff_Iutam08}. The existence question for the rate-independent model in terms of a weak reformulation is addressed in \cite{Neff_Chelminski07_disloc}. The rate-independent model with isotropic hardening is treated in \cite{Ebobisse_Neff09,Neff_Chelminski07_disloc}. The well-posedness of a rate-dependent variant without isotropic hardening is presented in \cite{NesenenkoNeff2011,NesenenkoNeff2012}.
First numerical results for a simplified rate-independent irrotational formulation (no plastic spin, symmetric plastic distortion $p$) are presented in \cite{Neff_Sydow_Wieners08}. In \cite{Lussardi08,Reddy06} well-posedness for a rate-independent model of Gurtin and Anand \cite{Gurtin05b} is shown under the decisive assumption that the plastic distortion is symmetric (the irrotational case), in which case one may really speak of a strain gradient plasticity model, since the full gradient acts on the symmetric plastic strain. 

Let us shortly revisit the modeling ingredients of the gradient plasticity model under consideration. 
This part does not contain new results but is added for clarity of exposition.
As usual in infinitesimal plasticity theory, the basic variables are the displacement $u:\Omega\to\R^3$ and the plastic distortion $p:\Omega\to \R^{3\times 3}$. We split the total displacement gradient $\nabla u$ into non-symmetric elastic and non-symmetric plastic distortions
\begin{align*}
  \nabla u=e+p\, .
\end{align*}
For invariance reasons, the elastic energy contribution may only depend on the symmetric elastic strains $\sym e=\sym (\nabla u-p)$. 
For more on the basic invariance questions related to this issue
dictating this type of behaviour, see \cite{Neff_Svendsen08,Neff_techmech07}. We assume as well plastic incompressibility $\Tr{p}=0$, as is usual.
The thermodynamic potential of our model is therefore written as
\begin{align} \label{energy}
\int\nolimits_\Omega &\Big(\underbrace{\mathbb{C}[x] (\sym (\nabla u-p)) (\sym (\nabla u-p))}_{\text{elastic energy}} \\
&+\underbrace{\frac{C_1[x] }{2} \, |\dev\sym p|^2}_{\text{kinematical hardening}}+\underbrace{\frac{C_2 }{2}|\Curl p|^2}_{\text{dislocation storage}} + \underbrace{u\cdot b}_{\text{external volume forces}}\Big)dx\non
\end{align}
The positive definite elasticity tensor $\mathbb{C}$ is able to represent the elastic anisotropy of the material. The plastic flow has the form
\begin{align}
  \partial_t p\in g(\sigma-C_1[x]  \dev\sym p-C_2\Curl\Curl p)\, ,
\end{align}
where $\sigma=\mathbb{C}[x]\sym (\nabla u-p)$ is the elastic symmetric Cauchy stress of the material and $g$ is a multivalued monotone flow function which is not necessary the subdifferential of a convex plastic potential (associative plasticity). This ensures the validity of the second law of 
thermodynamics, see \cite{Neff_Chelminski07_disloc}.

In this generality, our formulation comprises certain non-associative plastic flows in which the yield condition and the flow direction are independent and governed by distinct functions. Moreover, the flow function $g$ is supposed to induce a rate-dependent response as all materials are, in reality, rate-dependent.

Clearly, in the absence of energetic length scale effects (i.e. $C_2=0$), the $\Curl\Curl p$-term is absent. In general we assume that $g$ maps symmetric tensors to symmetric tensors. Thus, for $C_2=0$ the plastic distortion remains always symmetric and the model reduces to a classical plasticity model. Therefore, the energetic length scale is solely responsible for the plastic spin (the non-symmetry of $p$) in the model.

Regarding the boundary conditions necessary for the formulation of the higher order theory we assume that the so-called micro-hard boundary condition (see \cite{Gurtin05}) is specified, namely
\begin{align*}
       p\times n|_{\partial\Omega}=0.
\end{align*}
This is the correct boundary condition for tensor fields in $L^2_{\Curl}-$spaces
 which admits tangential traces.
We combine this with a new inequality extending Korn's inequality to incompatible tensor fields, namely
\begin{align}
\label{incompatible_korn}
\exists\, C=C(\Omega)>0\;& \forall \, p\in L^2_{\Curl}(\Omega, {\cal M}^3): \quad p\times n|_{\partial\Omega}=0:   \\
  & \underbrace{\|p\|_{L^2(\Omega)}}_{\text{plastic distortion}}\le C(\Omega)\, 
     \Big( \underbrace{\|\sym p\|_{L^2(\Omega)}}_{\text{plastic strain}}+ \underbrace{ \|\Curl p\|_{L^2(\Omega)}}_{\text{dislocation density}} \Big)\, .\notag
\end{align}
Here, the domain $\Omega$ needs to be {\bf sliceable}, i.e. cuttable into finitely many simply connected subdomains with Lipschitz boundaries. This inequality has been derived in \cite{Neff_Pauly_Witsch_cracad11, Neff_Pauly_Witsch_Korn_diff_forms_m2as12,Neff_Pauly_Witsch_Sbornik12
} and is precisely motivated by the well-posedness question for our model \cite{Neff_Chelminski07_disloc}. The inequality \eqref{incompatible_korn} expresses the fact that controlling the plastic strain $\sym p$ and the dislocation density $\Curl p$ in $L^2(\Omega)$ gives a control of the plastic distortion $p$ in $L^2(\Omega)$ provided the correct boundary conditions are specified: namely the micro-hard boundary condition. Since we assume that $\mathrm{tr}(p)=0$ (plastic incompressibility) the quadratic terms in the thermodynamic potential provide a control of the right hand side in \eqref{incompatible_korn}.

It is worthy to note that with $g$ only monotone and not necessarily a subdifferential the powerful energetic solution concept \cite{Mielke09,Lussardi08,Kratochvil10} cannot be applied. In our model we face the combined challenge of a gradient plasticity model based on the dislocation density tensor $\Curl p$ involving the plastic spin, a general non-associative monotone flow-rule and a rate-dependent response. 


\paragraph{Setting of the homogenization problem.}
Let $\Omega \subset \cR^3$ be an open bounded set, the set of material
points of the solid body, with a $C^2$-boundary and $Y\subset{\mathbb R}^3$ be a set 
having the paving property with respect to a basis $(b_1,b_2,b_3)$ defining the periods, 
a reference cell. By $T_e$ we denote a positive number (time of existence),
which can be chosen arbitrarily large,  and for $0 < t\leq T_e$
\begin{eqnarray}
{\Omega}_{t} = {\Omega} \times {(0, t)}. \non
\end{eqnarray} 
The sets, ${\cal M}^3$ and ${\cal S}^3$ denote the sets of all
$3 \times 3$--matrices and of all symmetric $3 \times 3$--matrices,
 respectively. Let $\mathfrak{sl}(3)$ be the set of all traceless $3 \times 3$--matrices,
 i.e. $$\mathfrak{sl}(3)=\{v\in{\cal M}^3\mid \Tr v=0\}.$$ Unknown in our small strain formulation are the displacement $u_\eta(x,t) \in
\cR^3$ of the material point $x$ at time $t$ and the non-symmetric infinitesimal
plastic distortion $p_\eta(x,t) \in\mathfrak{sl}(3)$. 

The model equations of the problem are
\begin{eqnarray}
- \di_x \sigma_\eta(x,t) &=&  b(x,t), \label{CurlPr1}
\\[1ex]
\sigma_\eta(x,t) &=& {\mathbb C}[x/\eta] ( \sym (\na_x u_\eta(x,t) - p_\eta(x,t) ) ),   \label{CurlPr2}
\\[1ex] 
\label{CurlPr3} \partial_t p_\eta(x,t) & \in & g \big(x/\eta,\Sigma^{\rm lin}_\eta(x,t)\big),  
\hspace{3ex} \Sigma^{\rm lin}_\eta=\Sigma^{\rm lin}_{e,\eta}+\Sigma^{\rm lin}_{\rm sh,\eta}
+\Sigma^{\rm lin}_{\rm curl,\eta},\label{microPr3} \\
\Sigma^{\rm lin}_{\rm e,\eta}&=&\sigma_\eta,
\hspace{1ex} \Sigma^{\rm lin}_{\rm sh,\eta}=-C_1[x/\eta]  \dev \sym p_\eta,
\hspace{1ex}  \Sigma^{\rm lin}_{\rm curl,\eta}=- C_2 \Curl\Curl p_\eta\, ,\non
\end{eqnarray}
which must be satisfied in $\Omega \times [0,T_e)$.  Here, $C_2\ge 0$ is a given material constant independent of $\eta$ and  $\Sigma^{\rm lin}_\eta$ is the infinitesimal Eshelby stress tensor driving the evolution of the plastic distortion $p_\eta$ and $\eta$ is a scaling parameter of the microstructure. 
The homogeneous initial condition and Dirichlet boundary condition are  
\begin{eqnarray} 
p_\eta(x,0) &=& 0, \hspace{7ex}  x \in \Omega,  \label{CurlPr4}   \\
p_\eta(x,t)\times n(x)&=&0,\hspace{7ex} (x,t) \in \partial\Omega \times
  [0,T_e), \label{CurlPr5}\\
u_\eta(x,t) &=& 0, \quad\quad\quad (x,t) \in \partial \Omega \times
  [0,T_e)\,,\label{CurlPr6}
\end{eqnarray}
where $n$ is a normal vector on the boundary $\partial\Omega$\footnote{Here, $v\times n$ with $v\in{\cal M}^3$ and $n\in\mathbb{R}^3$ denotes a row by column operation.}. For 
simplicity we consider only homogeneous boundary condition and we assume that 
the cell of periodicity is given by $Y=[0,1)^3$. Then, we assume that $C_1:Y\to\mathbb{R}$, a given material function, is measurable, periodic with the periodicity cell $Y$ and
satisfies the inequality
\begin{eqnarray}
C_1[y]\ge\alpha_1>0 \label{AssOnC1}
\end{eqnarray}
for all $y\in Y$ and some positive constant $\alpha_1$.
For every $y\in Y$ the
elasticity tensor ${\mathbb C}[y]: {\cal S}^3 \rightarrow {\cal S}^3$ is
linear symmetric and such that there
exist two positive constants $0 < \alpha \leq \beta$ satisfying
\begin{eqnarray}
    \alpha | \xi | ^{2} \leq {\mathbb C}_{ijkl}[ y ] \xi _{kl} \xi _{ij}
    \leq \beta | \xi | ^{2} \ \ \ \textrm{for} \ \textrm{any} \ \xi
    \in {\cal S}^3.\label{AssOnC}
\end{eqnarray}
We assume that the mapping 
$y\mapsto {\mathbb C}[y]:{\mathbb R}^3 \rightarrow {\cal S}^3$ is
measurable and periodic with the same periodicity cell $Y$. 
Due to the above assumption ($C_1>0$), the classical linear 
kinematic hardening is included in the model. Here, the nonlocal
 backstress contribution is given by the dislocation density motivated term 
$\Sigma^{\rm lin}_{\rm curl,\eta}=- C_2 \Curl\Curl p_\eta$ together with 
corresponding Neumann conditions.

For the model we require that the nonlinear constitutive mapping
$v\mapsto g(y,v):{\cal M}^3 \rightarrow 2^{\mathfrak{sl}(3)}$ is monotone for all  $y\in Y$, i.e.
it satisfies
\begin{eqnarray}
0 &\leq& (v_{1}-v_{2})\cdot (v^{*}_{1}-v^{*}_{2}), \label{monotype2}  
\end{eqnarray}
for all $v_i \in {\cal M}^3,\ v^{*}_i \in g(y,v_i),\ i=1,2$ and all  $y\in Y$.  We also require
that
\begin{eqnarray}
0 \in g(y,0), \hspace{3ex}\text{a.e.}\ y\in Y.  \label{monotype1} 
\end{eqnarray}
The mapping
$y\mapsto g(y,\cdot):{\mathbb R}^3 \rightarrow 2^{\mathfrak{sl}(3)}$ is periodic with the same periodicity cell $Y$. 
Given are the volume force $b(x,t)
\in \cR^3$ and the initial datum $
p^{(0)}(x) \in \mathfrak{sl}(3)$. 
\begin{Rem}
{\it It is well known that classical viscoplasticity (without gradient effects)
 gives rise to a well-posed problem. We extend this result to our formulation
  of rate-dependent gradient plasticity. The presence of the
  classical linear kinematic hardening in our model
  is related to $C_1>0$ whereas 
  the presence of the nonlocal gradient term is always related to $C_2>0$. }
\end{Rem}
The  development of the homogenization theory for the quasi-static initial boundary value
problem of monotone type in the classical elasto/visco-plasticity introduced by Alber in
\cite{Alb98} has started with the work  \cite{Alb00},
where the homogenized system of equations has been derived using the formal asymptotic
ansatz.  In the following work \cite{Alb03} Alber justified the formal asymptotic
ansatz for the case of positive definite
free energy\footnote{Positive definite energy corresponds to linear kinematic 
hardening behavior of materials.},  
employing the energy method of Murat-Tartar, yet only for local smooth solutions
of the homogenized problem. It is shown there that
the solutions of elasto/visco-plasticity problems can be approximated
in the $L^2(\Omega)-$norm by the smooth functions constructed from the solutions
of the homogenized problem. Later in \cite{Nes07}, under the assumption 
that the free energy is positive definite, it is proved that
the difference of the solutions of the microscopic problem 
and the solutions constructed from the homogenized problem, which both need not be smooth,
tends to zero in the $L^2(\Omega\times Y)-$norm, where $Y$ is the periodicity cell.
Based on the results obtained in \cite{Nes07}, in \cite{AlbNese09b} the convergence 
in $L^2(\Omega\times Y)$ is replaced by convergence in $L^2(\Omega)$. In the meantime,
for the rate-independent problems in plasticity similar results are
 obtained in \cite{Miel07} using the unfolding operator method 
 (see Section~\ref{periodicunfolding}) and methods of 
 energetic solutions due to Mielke. For special rate-dependent
models of monotone type, namely for rate-dependent generalized standard 
materials, the two-scale convergence of the solutions of the microscopic problem
to the solutions of the homogenized problem has been shown in \cite{Vis08b,Visintin08}. The homogenization of the Prandtl-Reuss model is performed in \cite{Schweizer10,Visintin08}.
In \cite{Nesenenko12a}
the author considered the rate-dependent problems of monotone type with constitutive functions
$g$, which need not  be subdifferentials, but which belong to the class of functions 
${\mathbb M}(\Omega, {\cal M}^3, q, \alpha, m)$ introduced
in Section~\ref{Homogenization}.
Using the unfolding operator method and in particular
the homogenization methods developed in \cite{Damlamian07},
for this class of functions the homogenized equations for the viscoplactic
 problems of monotone type are obtained in \cite{Nesenenko12a}.      

In the present work
the construction of the homogenization theory for the initial boundary value
problem (\ref{CurlPr1}) - (\ref{CurlPr6})  is based on the existence result derived
in \cite{NesenenkoNeff2012} (see Theorem~\ref{existMain}) and on the homogenization
techniques developed in \cite{Nesenenko12a} for classical viscoplasticity of monotone type.
 The existence result
in \cite{NesenenkoNeff2012} extends the well-posedness for infinitesimal dislocation based
gradient viscoplasticity with linear kinematic hardening from the subdifferential case
 (see \cite{NesenenkoNeff2011}) to general non-associative monotone plastic flows for 
 sliceable domains. In this work we also assume that the domain $\Omega$ is sliceable
 and that the monotone function $g:\mathbb{R}^3\times{\cal M}^3 \rightarrow 2^{\mathfrak{sl}(3)}$
 belongs to the class ${\mathbb M}(\Omega, {\cal M}^3, q, \alpha, m)$. 
 For sliceable domains $\Omega$, based on the inequality
(\ref{incompatible_korn}), we are able to derive then uniform estimates
for the solutions of (\ref{CurlPr1}) - (\ref{CurlPr6}) in Lemma~\ref{existLemma}.
Using the uniform estimates for the solutions of (\ref{CurlPr1}) - (\ref{CurlPr6}),
the unfolding operator method  and 
the homogenization techniques developed in \cite{Damlamian07,Nesenenko12a},
for the class of functions ${\mathbb M}(\Omega, {\cal M}^3, q, \alpha, m)$
 we obtain easily the homogenized equations for the original
 problem under consideration (see Theorem~\ref{HomoMain}). The distinguish feature of this work
 is that we use a variant of the unfolding operator due to Francu (see \cite{Francu_2010,Francu_Svanstedt_2012}) and not the one defined in \cite{Cioranescu08}. The
 modified unfolding operator helps to resolve the problems connecting with the need of 
 the careful treatment of the boundary layer in the definition of the unfolding operator in \cite{Cioranescu08}.
 To the best our knowledge this is the first 
 homogenization result obtained for the problem (\ref{CurlPr1}) - (\ref{CurlPr6}).
We note that similar homogenization results for the strain-gradient model of Fleck and Willis
 \cite{FleckWillis2004} are derived in \cite{Francfort2012,Giacomini2011,Hanke2011} 
 using the unfolding method 
 together with the $\Gamma$-convergence method in the rate-independent setting. 
 In \cite{Francfort2012} the authors, based on the assumption that the model 
 under consideration is of rate-independent type, are able to treat the case when 
 $C_2$ is a $Y$-periodic function as well. In the rate-independent setting this is
 possible due to the fact that the whole system (\ref{CurlPr1}) - (\ref{CurlPr6}) can be rewritten 
 as a standard variational inequality (see \cite{Han99}) and then the subsequant usage of 
 the techniques of the convex analysis enable the passage to the limit in the model equations.
 Contrary to this, in the rate-independent case this reduction to a single variational inequality
 is not possible and one is forced to use the monotonicity argument to study the asymptotic 
 behavior of the third term $\Sigma^{\rm lin}_{\rm curl,\eta}$ in (\ref{microPr3}).
\paragraph{Notation.} 
Suppose that $\Omega$ is a bounded domain with a $C^2$-boundary $\partial\Omega$.
Throughout the whole work we choose the numbers 
$q, q^*$ satisfying the
following conditions 
\[1 < q, q^* < \infty \ \  {\rm and}\ \
1/q + 1/q^* = 1,\] 
and $|\cdot|$ denotes a norm in ${\mathbb R}^k$. Moreover, the following notations are used in this work. 
The space $W^{m,q}(\Omega, \cR^k)$ with $q \in [1, \infty]$ consists
of all functions in $L^q(\Omega, \cR^k)$ with weak derivatives in
$L^q(\Omega, \cR^k)$ up to order $m$. If $m$ is not integer, then $W^{m,q}(\Omega, \cR^k)$ denotes 
the corresponding Sobolev-Slobodecki space. 
We set $H^m(\Omega, \cR^k)= W^{m,2}(\Omega, \cR^k)$. The norm in
$W^{m,q}(\Omega, {\mathbb R}^k)$ is denoted by $\| \cdot \|_{m,q,\Omega}$
($\| \cdot \|_{q}:=\| \cdot \|_{0,q,\Omega}$). 
The operator $\Gamma_0$ defined by
\[\Gamma_0: v\in W^{1,q}(\Omega, \cR^k)\mapsto W^{1-1/q,q}(\partial\Omega, {\mathbb R}^k)\]
  denotes the usual trace operator. The space $W^{m,q}_0(\Omega, \cR^k)$ with $q \in [1, \infty]$ consists
of all functions $v$ in $W^{m,q}(\Omega, \cR^k)$ with $\Gamma_0v=0$.
One can define the bilinear 
form on the product space
$L^{q}(\Omega, {\cal M}^3)$$\times$$L^{q^*}(\Omega, {\cal M}^3)$ by
\[
(\xi, \zeta )_{\Omega} = \int_{\Omega} \xi(x) \cdot \zeta(x) dx.
\]
 The space
\[L^q_{\Curl}(\Omega, {\cal M}^3)=\{v\in L^q(\Omega, {\cal M}^3)\mid
\Curl v\in L^q(\Omega, {\cal M}^3)\}\]
is a Banach space with respect to the norm
\[\|v\|_{q,\Curl}=\|v\|_{q}+\|\Curl v\|_{q}.\]
The well known result on the generalized trace operator (see \cite[Section II.1.2]{Sohr01})
 can be easily
adopted to the functions with values in ${\cal M}^3$. Then, according
to this result, there is a bounded operator 
$\Gamma_n$ on $L^q_{\Curl}(\Omega, {\cal M}^3)$ 
\[\Gamma_n: v\in L^q_{\Curl}(\Omega, {\cal M}^3)
\mapsto\big(W^{1-1/{q^*},q^*}(\partial\Omega, {\cal M}^3)\big)^{*}\]
with \[\Gamma_n v=v\times n\big|_{\partial\Omega}\ {\rm if}\
v\in C^1(\bar{\Omega}, {\cal M}^3),\]
where $X^*$ denotes the dual of a Banach space $X$.
 Next,
\[L^{q}_{\Curl,0}(\Omega, {\cal M}^3)=
\{w\in L^{q}_{\Curl}(\Omega,{\cal M}^3)\mid \Gamma_n(w)=0\}.\]
Let us define spaces $V^q(\Omega,{\cal M}^3)$ and $X^q(\Omega,{\cal M}^3)$ by
\[V^q(\Omega, {\cal M}^3)=\{v\in L^q(\Omega, {\cal M}^3)\mid
\di v, \Curl v\in L^q(\Omega, {\cal M}^3), \Gamma_n v=0\},\]
\[X^q(\Omega, {\cal M}^3)=\{v\in L^q(\Omega, {\cal M}^3)\mid
\di v, \Curl v\in L^q(\Omega, {\cal M}^3), \Gamma_0 v=0\},\]
which are Banach spaces with respect to the norm
\[\|v\|_{V^q}(\|v\|_{X^q})=\|v\|_{q}+\|\Curl v\|_{q}+\|\di v\|_{q}.\]
According to \cite[Theorem 2]{Kozono09}\footnote{This theorem has to 
be applied to each row of a function with values in ${\cal M}^3$ to obtain
 the desired result.} the spaces 
$V^q(\Omega, {\cal M}^3)$ and $X^q(\Omega,{\cal M}^3)$ are continuously imbedded into
$W^{1,q}(\Omega, {\cal M}^3)$. We define $V^q_\sigma(\Omega,{\cal M}^3)$ and
$X^q_\sigma(\Omega,{\cal M}^3)$ by
\[V^q_\sigma(\Omega,{\cal M}^3):=\{v\in V^q(\Omega,{\cal M}^3)\mid \di v =0\},\]
\[X^q_\sigma(\Omega,{\cal M}^3):=\{v\in X^q(\Omega,{\cal M}^3)\mid \di v =0\},\]
and denote by $V^q_{har}(\Omega,{\cal M}^3)$ and $X^q_{har}(\Omega,{\cal M}^3)$ the $L^q$-spaces of 
harmonic functions on $\Omega$ as 
\[V^q_{har}(\Omega,{\cal M}^3):=\{v\in V^q_\sigma(\Omega,{\cal M}^3)\mid \Curl v =0\},\]
\[X^q_{har}(\Omega,{\cal M}^3):=\{v\in X^q_\sigma(\Omega,{\cal M}^3)\mid \Curl v =0\},\]
Then the spaces 
$V^q_{har}(\Omega,{\cal M}^3)$ and $X^q_{har}(\Omega,{\cal M}^3)$ for every 
fixed $q$, $1<q<\infty$,
 coincides with the spaces $V_{har}(\Omega,{\cal M}^3)$ and $X_{har}(\Omega,{\cal M}^3)$ given by
\[V_{har}(\Omega,{\cal M}^3)=\{v\in C^\infty(\bar\Omega, {\cal M}^3)\mid
\di v=0,  \Curl v=0 \ {\rm with}\ v\cdot n=0 \ {\rm on}\ \partial\Omega\},\]
\[X_{har}(\Omega,{\cal M}^3)=\{v\in C^\infty(\bar\Omega, {\cal M}^3)\mid
\di v=0,  \Curl v=0 \ {\rm with}\ v\times n=0 \ {\rm on}\ \partial\Omega\},\]
respectively (see \cite[Theorem 2.1(1)]{Kozono09}).
The spaces $V_{har}(\Omega,{\cal M}^3)$ and $X_{har}(\Omega,{\cal M}^3)$
 are finite dimensional vector spaces (\cite[Theorem 1]{Kozono09}).

 We also define the space 
$Z^q_{\Curl}(\Omega, {\cal M}^3)$ by
\[Z^q_{\Curl}(\Omega, {\cal M}^3)=\{v\in L^q_{\Curl,0}(\Omega, {\cal M}^3)\mid
\Curl\Curl v\in L^q(\Omega, {\cal M}^3)\},\]
which is a Banach space with respect to the norm 
$$\|v\|_{Z^q_{\Curl}}=\|v\|_{q,\Curl}+\|\Curl\Curl v\|_{q}.$$ 

The space $W^{m,q}_{per}(Y,{\mathbb R}^k)$ denotes the Banach space of $Y$-periodic
functions in $W^{m,q}_{loc}({\mathbb R}^k,{\mathbb R}^k)$ equipped with the 
$W^{m,q}(Y,{\mathbb R}^k)$-norm.

For
functions $v$ defined on $\Omega \times [0,\infty)$ we denote by
$v(t)$ the mapping $x \mapsto v(x,t)$, which is defined on $\Omega$.
 The space $L^q(0,T_e; X)$ denotes the Banach space of all Bochner-measurable 
functions $u:[0,T_e)\to X$ such that $t\mapsto\|u(t)\|^q_X$ is integrable
on $[0,T_e)$. Finally, we frequently use the spaces $W^{m,q}(0,T_e;X)$, 
which consist of Bochner measurable functions having $q$-integrable weak
derivatives up to order $m$.
\section{Maximal monotone operators}\label{MonoOpers}

In this section we recall some basics about monotone 
and maximal monotone operators. For more details see 
\cite{Barb76,Hu97,Pas78}, for example.

 Let $V$ be a reflexive Banach space with the norm $\|\cdot\|$, $V^*$ 
be its dual space with the norm $\|\cdot\|_*$. The
brackets $\left< \cdot ,\cdot \right>$ denotes the dual pairing between
$V$ and $V^*$. Under $V$ we shall always mean a reflexive Banach space
throughout this section.
 For a multivalued mapping $A:V \to 2^{V^*}$ the sets
\[D(A)=\{v\in V\mid Av\not=\emptyset\} \] { and }
\[ Gr A=\{[v,v^*]\in V\times V^*\mid v\in D(A),\ v^*\in Av\}\]
are called the {\it effective domain} and the {\it graph} of $A$, respectively.

\begin{Def}
A mapping $A:V \to 2^{V^*}$ is called {\rm monotone} if and only if the
inequality holds
 \[\left< v^* - u^*, v - u \right> \ge 0  \ \ \ \ \forall \ [v,v^*],
 [u,u^*]\in Gr A.\]
 
A monotone mapping $A:V \to 2^{V^*}$ is called {\rm maximal monotone} iff the
inequality
 \[\left< v^* - u^*, v - u \right> \ge 0  \ \ \ \ \forall \ [u,u^*]\in Gr A\]
implies $[v,v^*]\in Gr A$.

A mapping $A:V \to 2^{V^*}$ is called {\rm generalized pseudomonotone} iff the
set $Av$ is closed, convex and bounded for all $v\in D(A)$ and for every pair of 
sequences $\{v_n\}$ and $\{v^*_n\}$ such that $v^*_n\in Av_n$, 
$v_n\rightharpoonup v_0$, $v^*_n\rightharpoonup v^*_0\in V^*$ and
 \[\limsup_{n\to\infty}\left< v^*_n, v_n - v_0 \right> \le 0, \]
we have that $[v_0,v_0^*]\in Gr A$ and $\left< v^*_n, v_n\right>\to\left< v^*_0, v_0\right>$.

A mapping $A:V \to 2^{V^*}$ is called {\rm strongly coercive} iff either 
$D(A)$ is bounded or $D(A)$
is unbounded and the condition
\[ \frac{\left< v^*, v - w \right>}{\|v\|} \to +\infty \ \ \ as \
\|v\| \to \infty, \ \ [v,v^*] \in Gr A,\]
is satisfied for each $w \in D(A)$. 
\end{Def}
It is well known (\cite[p. 105]{Pas78}) that if $A$ is a maximal monotone
operator, then for any $v\in D(A)$
the image $Av$ is a closed convex subset of $V^*$ and the graph
$Gr A$ is demi-closed.\footnote{A set $A\in V\times V^*$ is demi-closed
if $v_n$ converges strongly to $v_0$ in $V$ and $v^*_n$ converges weakly
to $v^*_0$ in $V^*$ (or $v_n$ converges weakly to $v_0$ in $V$ and $v^*_n$ converges strongly to $v^*_0$ in $V^*$) and $[v_n,v_n^*]\in Gr A$,
 then $[v,v^*]\in Gr A$} A maximal monotone operator is also generalized
 pseudomonotone (see \cite{Barb76,Hu97,Pas78}).  
\begin{Rem}\label{SubMax} {\rm We recall that the subdifferential of
 a lower semi-continuous and
convex function is maximal monotone (see \cite[Theorem 2.25]{Phel93}).}
\end{Rem}
\begin{Def}
The {\rm duality mapping} $J:V \to 2^{V^*}$ is defined by
\[J(v)=\{v^*\in V^*\ | \ \left< v^*, v \right> = \|v\|^2=\|v^*\|_*^2\ \} \]
for all $v \in V$. 
\end{Def}
Without loss of generality (due to Asplund's theorem) we can assume that
both $V$ and $V^*$ are strictly convex, i.e. that the unit ball in
 the corresponding space is strictly 
convex. In virtue of \cite[Theorem II.1.2]{Barb76}, the equation
\[J(v_\lambda - v)+\lambda Av_\lambda\ni 0\]
has a solution $v_\lambda\in D(A)$ for every $v\in V$ and $\lambda >0$ if
$A$ is maximal monotone. The solution is unique (see \cite[p. 41]{Barb76}).
\begin{Def}
Setting
\[v_\lambda=j^A_{\lambda}v \ \ \ { and} \ \ \ 
A_\lambda v=-\lambda^{-1}J(v_\lambda - v)\] 
we define two single valued operators: 
the {\rm Yosida approximation} $A_{\lambda}:V\to V^*$ and the
 {\rm resolvent} $j^A_{\lambda}:V\to D(A)$ with 
$D(A_{\lambda})=D(j^A_{\lambda})=V$.
\end{Def}
By the definition, one immediately sees that 
$A_\lambda v\in A\big( j^A_{\lambda}v \big)$. For the main properties
of the Yosida approximation we refer to \cite{Barb76,Hu97,Pas78} 
and mention only that both are continuous operators and that $A_\lambda$
is bounded and maximal monotone.

\paragraph{Convergence of maximal monotone graphs}

In the presentation of the next subsections we follow the work
\cite{Damlamian07}, where the reader can also find the proofs
of the results mentioned here. 

The derivation of the homogenized equations for the initial boundary
value problem (\ref{CurlPr1}) - (\ref{CurlPr6}) is based on the
notion of the convergence of the graphs of maximal monotone 
operators. 
According to  Brezis \cite{Brez73} and Attouch \cite{Attouch84},
 the convergence of the graphs of
 maximal monotone operators is defined as follows.
\begin{Def} Let $A^n$, $A :V \to 2^{V^*}$ be maximal monotone 
operators. The sequence $A^n$ converges to $A$ as $n\to\infty$,
($A^n\rightarrowtail A$), if for every $[v,v^*]\in  Gr A$ there exists
 a sequence $[v_n,v_n^*]\in Gr A^n$ such that
 $[v_n,v_n^*]\to [v,v^*]$ strongly in  $V \times V^{*}$ as $n\to\infty$.
\end{Def}
Obviously, if $A^n$ and $A$ are everywhere defined, continuous and monotone,
then the pointwise convergence, i.e. if for every $v\in V$, $A^n(v)\to A(v)$, 
implies the convergence of the graphs. The converse is true in 
finite-dimensional spaces.

The next theorem is the main mathematical tool in the derivation of
the homogenized equations for the problem (\ref{CurlPr1}) - (\ref{CurlPr6}). 
\begin{Theor}\label{convMaxMonGraph} 
Let $A^n$, $A :V \to 2^{V^*}$ be maximal monotone 
operators, and let $[v_n,v_n^*]\in Gr A^n$  and $[v,v^*]\in V \times V^{*}$. 
If, as $n\to\infty$, $A^n\rightarrowtail A$, $v_n\rightharpoonup v_0$, 
$v^*_n\rightharpoonup v^*_0\in V^*$ and
 \begin{equation}\label{convMaxMonGraphCondition}
 \limsup_{n\to\infty}\left< v^*_n, v_n\right> \le \left< v^*_0, v_0\right>,
 \end{equation}
then $[v_0,v_0^*]\in Gr A$ and 
$$\liminf_{n\to\infty}\left< v^*_n, v_n\right>=\left< v^*_0, v_0\right>.$$
\end{Theor}
\begin{proof}
 See \cite[Theorem 2.8]{Damlamian07}. 
\end{proof}
\begin{Rem}
We note that if a sequence $[v_n,v_n^*]\in Gr A^n$  in the definition of the  graph convergence of
 maximal monotone operators converges strongly to 
 some $[v,v^*]$ in $V \times V^{*}$ as $n\to\infty$, then 
 the condition (\ref{convMaxMonGraphCondition}) is satisfied and due to 
 Theorem~\ref{convMaxMonGraph} the limit $[v,v^*]$ belongs to the graph of the operator $A$.
 \end{Rem}
 The convergence of the graphs of multi-valued maximal monotone operators
can be equivalently stated in term of the pointwise convergence of 
the corresponding single-valued Yosida approximations and resolvents. 
\begin{Theor}\label{convMaxMonGrEquiv} 
Let $A^n$, $A :V \to 2^{V^*}$ be maximal monotone 
operators and $\lambda> 0$. The following statements are equivalent:
\begin{itemize}
       \item[(a)] $A^n\rightarrowtail A$ as $n\to\infty$;\vspace{-1ex}
       \item[(b)] for every $v\in V$, $j^{A^n}_{\lambda}v\to j^{A}_{\lambda}v$ 
as $n\to\infty$;\vspace{-1ex}
       \item[(c)] for every $v\in V$, ${A^n}_{\lambda}v\to {A}_{\lambda}v$ as $n\to\infty$;\vspace{-1ex}
     \item[(d)] $A^n_{\lambda}\rightarrowtail A_{\lambda}$ as $n\to\infty$.
       \end{itemize}
Moreover, the convergences $j^{A^n}_{\lambda}v\to j^{A}_{\lambda}v$ and 
${A^n}_{\lambda}v\to {A}_{\lambda}v$ are uniform on strongly compact 
subsets of $V$.
\end{Theor}
\begin{proof}
 See \cite[Theorem 2.9]{Damlamian07}. 
\end{proof}
\paragraph{Canonical extensions of maximal monotone operators.}\label{measurabilityMultis}
In this subsection we present briefly some facts about measurable
multi-valued mappings.
We assume that $V$, and hence $V^*$, is separable and denote 
the set of maximal monotone operators from $V$ to $V^*$
 by ${\mathfrak M}(V \times V^*)$. Further, let 
$(S, \Sigma(S), \mu)$ be a $\sigma-$finite $\mu-$complete
 measurable space. The notion of measurability 
for maximal monotone mappings can be defined in terms
of the measurability for appropriate single-valued mappings.
\begin{Def} 
A function $A:S\to{\mathfrak M}(V \times V^*)$ 
 is measurable iff for every $v\in E$, $x\mapsto j^{A(x)}_{\lambda}v$ is measurable
\end{Def}
For further reading on  measurable
multi-valued mappings we refer the reader to \cite{Castaing77,Damlamian07,Hu97,Pankov97}.

Given a mapping $A:S\to{\mathfrak M}(V \times V^*)$, one can define
 a monotone graph from $L^p(S,V)$ to $L^q(S,V^*)$, where
$1/p + 1/q = 1$, as follows:
\begin{Def}\label{CanExtension}
 Let $A:S\to{\mathfrak M}(V \times V^*)$, the canonical extension of $A$
 from $L^p(S,V)$ to $L^q(S,V^*)$, where
$1/p + 1/q = 1$, is defined by:
\[Gr {\cal A} = \{[v, v^*] \in L^p(S,V)\times L^q(S,V^*)\mid
 [v(x), v^*(x)] \in Gr A(x)\ for\ a.e.\ x\in S\}.\]
\end{Def}
Monotonicity of ${\cal A}$ defined in Definition~\ref{CanExtension}
 is obvious, while its maximality follows from the next proposition.
\begin{Prop}
 Let $A:S\to{\mathfrak M}(V \times V^*)$ be measurable. 
If $Gr {\cal A}\not= \emptyset$, then ${\cal A}$ is maximal monotone.
\end{Prop}
\begin{proof}
 See \cite[Proposition 2.13]{Damlamian07}. 
\end{proof}
We have to point out here that the maximality of $A(x)$ 
for almost every $x\in S$ does not imply
 the maximality of ${\cal A}$ as 
the latter can be empty (\cite{Damlamian07}): 
$S = (0, 1),$ and $Gr A(x) = \{[v,v^*]\in {\mathbb R}\times
 {\mathbb R}\mid v^*= x^{-1/q}\}.$

For given mappings $A, A^n:S\to{\mathfrak M}(V \times V^*)$
 and their canonical extensions ${\cal A}, {\cal A}^n$, one can ask
whether the pointwise convergence 
$A^n(x)\rightarrowtail A(x)$ implies the convergence of
 the graphs of the corresponding canonical extensions
${\cal A}^n\rightarrowtail{\cal A}$. The answer is given by the next theorem.
\begin{Theor}\label{GraphConvCanExten}
Let $A, A^n:S\to{\mathfrak M}(V \times V^*)$ be measurable. Assume
\begin{itemize}
       \item[(a)] for almost every $x\in S$, $A^n(x)\rightarrowtail A(x)$ 
       as $n\to\infty$,\vspace{-1ex}
       \item[(b)] ${\cal A}$ and ${\cal A}^n$ are maximal monotone,\vspace{-1ex}
       \item[(c)] there exists $[\alpha_n,\beta_n]\in Gr {\cal A}^n$
        and $[\alpha,\beta]\in L^p(S,V)\times L^q(S,V^*)$
       such that $[\alpha_n,\beta_n]\to [\alpha,\beta]$ strongly
         in $L^p(S,V)\times L^q(S,V^*)$ as $n\to\infty$,
       \end{itemize}
 then ${\cal A}^n\rightarrowtail{\cal A}$.
\end{Theor}
\begin{proof}
 See \cite[Proposition 2.16]{Damlamian07}. 
\end{proof}
We note that assumption (c) in Theorem~\ref{GraphConvCanExten}
 can not be dropped in virtue of Remark 2.16 in \cite{Damlamian07}.

\section{The periodic unfolding}\label{periodicunfolding}

The derivation of the homogenized problem for (\ref{CurlPr1}) - (\ref{CurlPr6})
 is based on the periodic unfolding operator method. In 1990, Arbogast, Douglas and Hornung used
 a so-called dilation operator to study the homogenization of double-porosity periodic medium in
 \cite{Arbogast_Hornung_1990} (see \cite{Casado_Diaz_2000,Casado_Diaz_2001} for further applications of the method). 
 This idea has been extended and further developed in \cite{Cioranescu02} for two-scale and
 multi-scale homogenization under the name of "unfolding method". Nowadays there exists 
 an extensive literature concerning the applications and extensions of the unfolding operator method. 
 We recommend an interested reader to have a look into 
 the following survey papers \cite{Cioranescu12,Cioranescu08} and in the literature cited there. 
 We recall briefly the definition of the unfolding operator due to
 Cioranescu, Damlamian and Griso (\cite{Cioranescu02,Cioranescu08}): 

Let $\Omega\subset{\mathbb R}^3$ be an open set and $Y=[0,1)^3$. Let $(e_1,e_2,e_3)$ denote the standard basis in ${\mathbb R}^3$.
For $z\in{\mathbb R}^3$, $[z]_Y$ denotes
a linear combination $\sum_{j=1}^3 d_je_j$ with $\{d_1,d_2,d_3\}\in{\mathbb Z}$  such 
that $z - [z]_Y$ belongs to $Y$, and set
\[\{z\}_Y:=z - [z]_Y\in Y \hspace{2ex} v\in{\mathbb R}^3. \]
Then, for each $x\in{\mathbb R}^3$, one has
\[x=\eta\left(\left[\frac{x}{\eta}\right]_Y + y\right).\]
We use the following notations:
\[ \Xi_\eta=\{\xi\in{\mathbb Z}^k\mid\eta(\xi + Y)\subset\Omega\},\ \
\hat\Omega_\eta=\inter\left\{\bigcup_{\xi\in\Xi_\eta}\left(\eta\xi
 + \eta\overline Y\right)\right\},\ \
 \Lambda_\eta=\Omega\setminus\hat\Omega_\eta.\]
The set $\hat\Omega_\eta$  is the largest union of
$\eta(\xi + \overline Y)$ cells ($\xi\in{\mathbb Z}^3$) included 
in $\Omega$, while $\Lambda_\eta$ is
the subset of $\Omega$ containing the parts from $\eta(\xi + \overline Y)$
cells intersecting the boundary $\partial\Omega$.
\begin{Def}\label{UnfoldingOper}
 Let $Y$ be a reference cell, $\eta$ be a positive number and a map $v:\Omega\to {\mathbb R}^k$. 
The unfolding operator
${\cal T}_\eta(v):\Omega\times Y\to {\mathbb R}^k$ is defined by
\begin{eqnarray}\label{Unfolding}
\left({\cal T}_\eta v\right)(x,y):=\begin{cases}
    v\left(\eta\left[\frac{x}{\eta}\right]_Y + \eta y\right), & a.e.\ 
                                   (x,y)\in\hat\Omega_\eta\times Y,\\ 
    0, & a.e.\ (x,y)\in\Lambda_\eta\times Y.
\end{cases}
\end{eqnarray}
\end{Def}
From Definition~\ref{UnfoldingOper} it easily follows that,
 for $q\in [1,\infty[$, the operator ${\cal T}_\eta$ is linear and continuous
from $L^q(\Omega,{\mathbb R}^k)$ to $L^q(\Omega\times Y,{\mathbb R}^k)$ and that
or every $\phi$ in $L^1(\Omega,{\mathbb R}^k)$ one has 
\begin{eqnarray}\label{integration}
\frac1{|Y|}\int_{\Omega\times Y}{\cal T}_\eta(\phi)(x,y)dxdy=
          \int_{\hat\Omega_\eta}\phi(x)dx
\end{eqnarray}          
and
  $$\left|\int_{\hat\Omega_\eta}\phi(x)dx-
       \frac1{|Y|}\int_{\Omega\times Y}{\cal T}_\eta(\phi)(x,y)dxdy\right|
          \le\int_{\Lambda_\eta}|\phi(x)|dx.$$
Obviously, if $\phi_\eta\in L^1(\Omega,{\mathbb R}^k)$ satisfies
\begin{eqnarray}\label{uci}
\int_{\Lambda_\eta}|\phi_\eta(x)|dx\to 0,
\end{eqnarray}
then
\[\int_{\Omega}\phi_\eta(x)dx-
       \frac1{|Y|}\int_{\Omega\times Y}{\cal T}_\eta(\phi_\eta)(x,y)dxdy\to 0.\]
In \cite{Cioranescu08}, each sequence $\phi_\eta$ fulfilling \eq{uci} has been called 
the sequence satisfying unfolding criterion for integrals and this has been denoted as follows
\[\int_{\Omega}\phi_\eta(x)dx\overset{{\cal T}_\eta}{\simeq}\frac1{|Y|}\int_{\Omega\times Y}{\cal T}_\eta(\phi_\eta)(x,y)dxdy. \]
The fact, that we can not consider the integration on the righthand side in (\ref{integration}) 
over the whole domain $\Omega$ and have to establish the validity of the unfolding criterion 
for integrals for a sequence of functions, can cause some difficulty due to 
the necessity of the careful treatment of the boundary layer in (\ref{uci}). In 
\cite{Francu_2010,Francu_Svanstedt_2012} this problem 
has been resolved by extending the unfolding operator by the identity:
\begin{eqnarray}\label{NewUnfolding}
\left({\cal T}_\eta v\right)(x,y):=\begin{cases}
    v\left(\eta\left[\frac{x}{\eta}\right]_Y + \eta y\right), & a.e.\ 
                                   (x,y)\in\hat\Omega_\eta\times Y,\\ 
    v(x), & a.e.\ (x,y)\in\Lambda_\eta\times Y.
\end{cases}
\end{eqnarray}
The unfolding operator in \eq{NewUnfolding} conserves the integral, i.e. every $\phi$ in $L^1(\Omega,{\mathbb R}^k)$ one has 
\[
\frac1{|Y|}\int_{\Omega\times Y}{\cal T}_\eta(\phi)(x,y)dxdy=
          \int_{\Omega}\phi(x)dx,
\]
which implies that it is an isometry between $L^q(\Omega,{\mathbb R}^k)$ and
 $L^q(\Omega\times Y,{\mathbb R}^k)$. In case of a general bounded domain $\Omega$, i.e. when 
 $|\Lambda_\eta|>0$ and $|\Lambda_\eta|\to0$, 
 both definitions of the unfolding operator \eq{Unfolding}
and \eq{NewUnfolding} are equivalent for the sequences, which are bounded in 
$L^q(\Omega,{\mathbb R}^k)$. For the sequences, which are unbounded in 
$L^q(\Omega,{\mathbb R}^k)$, the definitions differ (see \cite[Section 4]{Francu_Svanstedt_2012}).
Since in this work we are dealing only with bounded sequence, we shall not introduce a new notation 
for the unfolding operator \eq{NewUnfolding} and use the results in \cite{Cioranescu08}, 
which are proved for bounded sequences in $L^q(\Omega,{\mathbb R}^k)$ and
 the unfolding operator defined by \eq{Unfolding}. 
\begin{Prop}\label{UnfoldWeakConv}
Let $q$ belong to $[1,\infty[$.
\begin{itemize}
    \item[(a)] For any $v\in L^q(\Omega,{\mathbb R}^k)$, 
${\cal T}_\eta(v)\to v$ strongly in $L^q(\Omega\times Y,{\mathbb R}^k)$,
    \item[(b)]  Let $v_\eta$ be a bounded sequence
   in $L^q(\Omega,{\mathbb R}^k)$ such that 
      $v_\eta\to v$ strongly in $L^q(\Omega,{\mathbb R}^k)$, then
      \[{\cal T}_\eta(v_\eta)\to v, \ \ strongly \ in\ 
           L^q(\Omega\times Y,{\mathbb R}^k).\]
    \item[(c)] For every relatively weakly compact sequence $v_\eta$ 
         in $L^q(\Omega,{\mathbb R}^k)$, 
             the corresponding ${\cal T}_\eta(v_\eta)$ 
         is relatively weakly compact in $L^q(\Omega\times Y,{\mathbb R}^k)$. 
      Furthermore,
       if
\begin{eqnarray}
{\cal T}_\eta(v_\eta)\rightharpoonup \hat{v} \ \ in \ 
L^q(\Omega\times Y,{\mathbb R}^k),\non
\end{eqnarray}
then
\begin{eqnarray}
v_\eta\rightharpoonup \frac1{|Y|}\int_Y\hat{v}dy \ \ in \ 
L^q(\Omega,{\mathbb R}^k).\non
\end{eqnarray}
    \end{itemize}
\end{Prop}
\begin{proof}
 See \cite[Proposition 2.9]{Cioranescu08}. 
\end{proof}
Next results present some properties of the
restriction of the unfolding operator to the space 
$W^{1,q}(\Omega,{\mathbb R}^k)$.
\begin{Prop}
Let $q$ belong to $]1,\infty[$. 
 Let $v_\eta$ converge weakly in $W^{1,q}(\Omega,{\mathbb R}^k)$ 
to $v$. Then
\[{\cal T}_\eta(v_\eta)\rightharpoonup {v} \ \ in \ 
L^q(\Omega, W^{1,q}_{per}(Y,{\mathbb R}^k)).\]
\end{Prop}
\begin{proof}
 See \cite[Corollary 3.2, Corollary 3.3]{Cioranescu08}. 
\end{proof}
\begin{Prop}\label{UnfoldGradient}
Let $q$ belong to $]1,\infty[$. 
 Let $v_\eta$ converge weakly in $W^{1,q}(\Omega,{\mathbb R}^k)$ 
to some $v$. Then,
up to a subsequence, there exists some 
$\hat{v}\in L^q(\Omega, W^{1,q}_{per}(Y,{\mathbb R}^k))$
such that
\[{\cal T}_\eta(\na v_\eta)\rightharpoonup \na{v}+\na_y\hat{v} \ \ in 
\ L^q(\Omega\times Y,{\mathbb R}^k).\]
 \end{Prop}
\begin{proof}
 See \cite[Theorem 3.5]{Cioranescu08}. 
\end{proof}
The last proposition can be generalized to $W^{m,q}(\Omega,{\mathbb R}^k)$-spaces
with $m\ge 1$.
\begin{Prop}\label{UnfoldGradientOrderM}
Let $q$ belong to $]1,\infty[$ and $m\ge1$. 
 Let $v_\eta$ converge weakly in $W^{m,q}(\Omega,{\mathbb R}^k)$ 
to some $v$. Then,
up to a subsequence, there exists some 
$\hat{v}\in L^q(\Omega, W^{m,q}_{per}(Y,{\mathbb R}^k))$
such that
\begin{eqnarray}
&&{\cal T}_\eta(D^l v_\eta)\rightharpoonup D^l{v} \ \ in 
\ L^q(\Omega,W^{m-l,q}(Y,{\mathbb R}^k))\ for \ |l|\le m-1, \non\\
&&{\cal T}_\eta(D^l v_\eta)\rightharpoonup D^l{v}+D^l_y\hat{v} \ \ in 
\ L^q(\Omega\times Y,{\mathbb R}^k)\ for \ |l|=m\non
\end{eqnarray}
\end{Prop}
\begin{proof}
 See \cite[Theorem 3.6]{Cioranescu08}. 
\end{proof}
For a multi-valued function 
$h\in {\mathbb M}(\Omega,{\mathbb R}^k,\alpha,m)$\footnote{The class of
functions $h\in {\mathbb M}(\Omega,{\mathbb R}^k,\alpha,m)$ 
is defined in Definition~\ref{CoercClass}.} we define
the unfolding operator as follows.
\begin{Def}\label{UnfoldingOperMulti}
 Let $Y$ be a periodicity cell, $\eta$ be a positive number and a map 
$h\in {\mathbb M}(\Omega,{\mathbb R}^k,p,\alpha,m)$. 
The unfolding operator
${\cal T}_\eta(h):\Omega\times Y\times{\mathbb R}^k\to 2^{{\mathbb R}^k}$ is defined by
\[{\cal T}_\eta(h)(x,y,v):=\begin{cases}
    h\left(\eta\left[\frac{x}{\eta}\right]_Y + \eta y,v\right), & a.e.\ 
                                   (x,y)\in\hat\Omega_\eta\times Y,\; v\in{\mathbb R}^k,\\ 
    |v|^{p-2}v, & a.e.\ (x,y)\in\Lambda_\eta\times Y,\; v\in{\mathbb R}^k.
\end{cases}\]
\end{Def}
Obviously, by its definition the unfolding operator of a multi-valued function
from ${\mathbb M}(\Omega,{\mathbb R}^k,\alpha,m)$ belongs to the set 
${\mathbb M}(\Omega\times Y,{\mathbb R}^k,\alpha,m)$.

We note that the periodic unfolding method described above is an
alternative to the two-scale convergence method introduced in 
\cite{Nguetseng89} and further developed in \cite{Alla92}. More precisely,
the two-scale convergence of a bounded sequence $v_\eta$
 in $L^p(\Omega,{\mathbb R}^k)$
is equivalent to the weak convergence of the corresponding unfolded
sequence ${\cal T}_\eta(v_\eta)$ in $L^p(\Omega\times Y,{\mathbb R}^k)$
(see \cite[Proposition 2.14]{Cioranescu08} or \cite{Francu_2010,Francu_Svanstedt_2012,Lenczner_1997}).

\section{Unfolding the $\Curl\Curl$-operator}

Our method is based  on the Helmholtz-Weyl decomposition for vector fields in general
$L^q$-spaces over a domain $\Omega$ with a $C^2$-boundary $\partial\Omega$. It turns out (see \cite[Theorem 2.1(2)]{Kozono09})
that the following theorem holds.
\begin{Theor}\label{HelmholtzZerlegungTh} 
Let $1<q<\infty$.
Every $v\in L^q(\Omega,{\mathbb R}^3)$ can be uniquely decompose as
\begin{eqnarray}\label{HelmholtzZerlegung}
v=h+\Curl w+\na z,
\end{eqnarray}
where $h\in X^q_{har}(\Omega,\cR^3)$, $w\in V^q_{\sigma}(\Omega,\cR^3)$ and 
$z\in W^{1,q}(\Omega,\cR^3)$, and the triple $(h,w,z)$ satisfies the inequality
\begin{eqnarray}\label{EstimateHelmholtzZerlegungMain}
\|h\|_{q}+\|w\|_{1,q,\Omega}+\|z\|_{1,q,\Omega}\le C\|v\|_q,
\end{eqnarray}
where $C$ is a constant depending on $\Omega$ and $q$.
 If there is another triple of functions $(\tilde h, \tilde w, \tilde z)$
such that $v$ can be written in the form
\begin{eqnarray}
v=\tilde h+\Curl\tilde w+\na\tilde z,\non
\end{eqnarray}
with $\tilde h\in X^q_{har}(\Omega,\cR^3)$, $\tilde w\in V^q_{\sigma}(\Omega,\cR^3)$
and 
$\tilde z\in W^{1,q}(\Omega,\cR^3)$, then it holds
\[h=\tilde h, \ \ \Curl w=\Curl \tilde w, \ \ \na z=\na \tilde z.\]
\end{Theor}
\begin{Rem}
If $L$ denotes the dimension of $V_{har}(\Omega,\cR^3)$, i.e.
 ${\rm dim} V_{har}(\Omega,\cR^3)=L$,
and $\{\phi_1,...,\phi_L\}$ is a basis of $V_{har}(\Omega,\cR^3)$, then 
it holds $V^{q}(\Omega,\cR^3)\subset W^{1,q}(\Omega,\cR^3)$ with the estimate
\begin{eqnarray}
\label{estimateKOZONOappendix}
\|v\|_{q}+\|\nabla v\|_{q}\le C(\|\Curl v\|_{q}+\|\di v\|_{q}+\sum_{i=1}^{L}|(v,\phi_i)|)\non
\end{eqnarray}
for all $v\in V^{q}(\Omega,\cR^3)$, where $C=C(\Omega,q)$ 
(\cite[Theorem 2.4(2)]{Kozono09}). The proof of the inequality 
(\ref{estimateKOZONOappendix}) with $\sum_{i=1}^{L}|(v,\phi_i)|$ 
replaced by $\|v\|_q$ is performed in \cite[Lemma 4.5]{Kozono09}
(for $q=2$ it can be found in \cite[Theorem VII.6.1]{Duvaut76}).
If we assume that the boundary $\partial\Omega$ has $L+1$ smooth 
connected  components $\Gamma_0, \Gamma_1, ...,
\Gamma_L$ such that $\Gamma_1, ...,\Gamma_L$ lie inside $\Gamma_0$ 
with $\Gamma_i\cap\Gamma_j=\emptyset$ for $i\not=j$ and 
\begin{eqnarray}
\partial\Omega=\cup_{i=0}^{L}\Gamma_i,\non
\end{eqnarray}
then it holds (\cite[Appendix A]{Kozono09}) 
\[{\rm dim} V_{har}(\Omega,\cR^3)=L.\]
\end{Rem}
If the function $v$ in (\ref{HelmholtzZerlegung}) is more regular, then the function $w$
can be chosen from a better space as the next theorem shows.
\begin{Theor}\label{HelmholtzZerlegungThReg}
Let $1<q<\infty$. Assume that decomposition (\ref{HelmholtzZerlegung}) holds. If, additionally
$v\in Z^q_{\Curl}(\Omega,{\mathbb R}^3)$,  then $w$ in (\ref{HelmholtzZerlegung}) 
can be chosen from 
$W^{3,q}(\Omega, {\mathbb R}^3)\cap V^q_{\sigma}(\Omega,{\mathbb R}^3)$ satisfying
the estimate
\begin{eqnarray}\label{EstimateHelmholtzZerlegung}
\|w\|_{3,q,\Omega}\le C(\|\Curl v\|_{1,q,\Omega}+\|v\|_q),
\end{eqnarray}
where $C$ is a constant depending on $\Omega$ and $q$.
\end{Theor}
\begin{proof} For $v\in L^q_{\Curl}(\Omega, {\mathbb R}^3)$ this result is proved in
 \cite{Kozono09a}.
For $v\in Z^q_{\Curl}(\Omega,{\mathbb R}^3)$ the proof runs the same lines. We repeat them.

As it is shown in \cite[Lemma 4.2(2)]{Kozono09}, we can choose the function 
$w\in V^q_{\sigma}(\Omega,{\mathbb R}^3)$ satisfying the equation
\begin{eqnarray}\label{EstimateHelmholtzEq1}
(\Curl w,\Curl \psi)_\Omega=(v,\Curl \psi)_\Omega, \ \ {\rm for}\ {\rm all}\ \psi\in 
V^{q^*}_{\sigma}(\Omega,{\mathbb R}^3)
\end{eqnarray}
with the estimate 
\begin{eqnarray}\label{EstimateHelmholtzEq2}
\|w\|_{1,q,\Omega}\le C\|v\|_q,
\end{eqnarray}
where $C$ depends only on $\Omega$ and $q$. Since $\di w=0$ in $\Omega$ and
$v\in Z^q_{\Curl}(\Omega,{\mathbb R}^3)$, it follows from (\ref{EstimateHelmholtzEq1})
that $-\Delta w=\Curl v$ in the sense of distributions, and we may regard $w$ as a weak
solution of the following boundary value problem
\begin{eqnarray}\label{EstimateHelmholtzEq3}
-\Delta w&=&\Curl v, \hspace{3ex} {\rm in}\ \Omega, \\
\di w&=&0,\hspace{8ex} {\rm on}\ \partial\Omega, \\
w\cdot n&=&0,\hspace{8ex} {\rm on}\ \partial\Omega.
\end{eqnarray}
Since $\Curl v\in W^{1,q}(\Omega, {\mathbb R}^3)$, it follows from 
\cite[Lemma 4.3(1)]{Kozono09} and the classical theory of Agmon, Douglas and
Nirenberg \cite{AgmonDouglisNirenberg}
 that the solution $w$ of the homogeneous boundary value problem
(\ref{EstimateHelmholtzEq3}) belongs to $W^{3,q}(\Omega, {\mathbb R}^3)$ and
the estimate 
\begin{eqnarray}\label{EstimateHelmholtzEq4}
\|w\|_{3,q,\Omega}\le C(\|\Curl v\|_{1,q,\Omega}+\|w\|_q),
\end{eqnarray}
is valid with the constant $C$ dependent of $\Omega$ and $q$. Due to 
(\ref{EstimateHelmholtzEq2}), the estimate (\ref{EstimateHelmholtzEq4}) implies
(\ref{EstimateHelmholtzZerlegung}). This completes the proof.
\end{proof}
Now we can state the main result of this section.
\begin{Theor}\label{HelmholtzZerlegungThMain}
Let $1<q<\infty$. Suppose that sequence $v_\eta$  is weakly compact in 
$Z^{q}_{\Curl}(\Omega,{\mathbb R}^3)$. Then there exist  
\begin{eqnarray}
&&
v\in Z^{q}_{\Curl}(\Omega,{\mathbb R}^3),\ \ v_0\in L^{q}(\Omega\times Y,{\mathbb R}^3)
 \  with \ \Curl_y v_0=0,\non\\
&&
 v_1\in L^{q}(\Omega, W^{2,q}_{per}(Y, {\mathbb R}^3))
 \  with \ \di_y v_1=0,\non
\end{eqnarray}
 such that
\begin{eqnarray}
&&v_\eta\rightharpoonup v \ \ in \ Z^{q}_{\Curl}(\Omega,{\mathbb R}^3),
\label{UnfoldingCurl1}\\[1ex]
&&{\cal T}_\eta(v_\eta)\rightharpoonup v_0\ \ in \ 
L^{q}(\Omega\times Y,{\mathbb R}^3),\label{UnfoldingCurl2}\\[1ex]
&&{\cal T}_\eta(\Curl v_\eta)\rightharpoonup \Curl v\ \ in \ 
L^{q}(\Omega, W^{1,q}_{per}(Y,{\mathbb R}^3)),\label{UnfoldingCurl3}\\[1ex]
&&{\cal T}_\eta(\Curl\Curl v_\eta)\rightharpoonup \Curl\Curl v+\Curl_y\Curl_y v_1 \ in \ 
L^{q}(\Omega\times Y,{\mathbb R}^3).\label{UnfoldingCurl4}
\end{eqnarray}
Moreover, $v(x)=\int_Yv_0(x,y)dy$.
\end{Theor}
\begin{proof} Convergence (\ref{UnfoldingCurl2}) and the last statement of the theorem
follow from Proposition~\ref{UnfoldWeakConv}(c). Convergence (\ref{UnfoldingCurl1}) is
obvious.  Next, we prove convergences (\ref{UnfoldingCurl3}) and (\ref{UnfoldingCurl4}). 
According to Theorem~\ref{HelmholtzZerlegungTh}, there exist 
$h_\eta\in X^q_{har}(\Omega,\cR^3)$, $w_\eta\in V^q_{\sigma}(\Omega,\cR^3)$ and 
$z_\eta\in W^{1,q}(\Omega,\cR^3)$ satisfying the inequality
\begin{eqnarray}\label{EstimateUnfoldingCurl1}
\|h_\eta\|_{q}+\|w_\eta\|_{1,q,\Omega}+\|z_\eta\|_{1,q,\Omega}\le C\|v_\eta\|_q
\end{eqnarray}
with the constant $C$ independent of $\eta$, and such that
\begin{eqnarray}\label{EstimateUnfoldingCurl2}
v_\eta=h_\eta+\Curl w_\eta+\na z_\eta.
\end{eqnarray}
Moreover, due to Theorem~\ref{HelmholtzZerlegungThReg}, $w_\eta$ in 
(\ref{EstimateUnfoldingCurl2}) enjoys the inequality
\begin{eqnarray}\label{EstimateUnfoldingCurl3}
\|w_\eta\|_{3,q,\Omega}\le C(\|\Curl v_\eta\|_{1,q,\Omega}+\|v_\eta\|_q)
\end{eqnarray}
with the constant $C$ independent  of $\eta$. Therefore,
the weak compactness
of $v_\eta$ in  $Z^{q}_{\Curl}(\Omega,{\mathbb R}^3)$ and 
(\ref{EstimateUnfoldingCurl3}) imply that $w_\eta$ is weakly compact in 
$W^{3,q}(\Omega,\cR^3)$. Thus, in virtue of Proposition~\ref{UnfoldGradientOrderM}
we conclude that there exist 
\[w\in W^{3,q}(\Omega,\cR^3) \ \ {\rm and}\ \ w_1\in L^q(\Omega,W^{3,q}_{per}(Y,\cR^3))\]
such that
\begin{eqnarray}
&&{\cal T}_\eta(D^l w_\eta)\rightharpoonup D^l{w} \ \ {\rm in} 
\ L^q(\Omega,W^{3-l,q}(Y,{\mathbb R}^k))\ {\rm for} \ |l|\le 2, \label{InequalityNeededLater1}\\
&&{\cal T}_\eta(D^l w_\eta)\rightharpoonup D^l{w}+D^l_y w_1 \ \ {\rm in} 
\ L^q(\Omega\times Y,{\mathbb R}^k)\ {\rm for}  \ |l|=3.\label{InequalityNeededLater2}
\end{eqnarray}
Since $\Curl v_\eta=\Curl\Curl w_\eta$ and $\Curl v=\Curl\Curl w$, we get that
\begin{eqnarray}
&&{\cal T}_\eta(\Curl v_\eta)\rightharpoonup \Curl v\ \ {\rm in} \ 
L^{q}(\Omega, W^{1,q}_{per}(Y,{\mathbb R}^3)),\non\\[1ex]
&&{\cal T}_\eta(\Curl\Curl v_\eta)\rightharpoonup \Curl\Curl v+\Curl_y\Curl_y v_1\ \ {\rm in} \ 
L^{q}(\Omega\times Y,{\mathbb R}^3).\non
\end{eqnarray}
It is left to prove that the condition $\Curl_y v_0=0$ is valid\footnote{The proof of this result is due to an unknown reviewer of the manuscript.}. To this end, we note first that the additive decomposition
(\ref{EstimateUnfoldingCurl2}) implies
\begin{eqnarray}\label{EstimateUnfoldingCurl3}
{\cal T}_\eta(v_\eta)={\cal T}_\eta(h_\eta)+{\cal T}_\eta(\Curl w_\eta)+{\cal T}_\eta(\na z_\eta).
\end{eqnarray}
Since the function $h_\eta$ belongs to the space of smooth functions $X^q_{har}(\Omega,\cR^3)$,
up to a subsequence, the sequence $h_\eta$ converges strongly to 
a function $h$ in $L^q(\Omega, \mathbb{R}^3)$. This provides that 
$${\cal T}_\eta(h_\eta)\to h\ \ {\rm in} \ L^{q}(\Omega\times Y,{\mathbb R}^3).$$
Next, the weak compactness of $w_\eta$ in $W^{3,q}(\Omega,\cR^3)$ together with 
the convergence (\ref{InequalityNeededLater1}) and Rellich's theorem
guarantee  that
$${\cal T}_\eta(\Curl w_\eta)\to \Curl w\ \ {\rm in} \ L^{q}(\Omega\times Y,{\mathbb R}^3).$$
Proposition~\ref{UnfoldGradientOrderM} applied to the gradient of $z_\eta$ implies that there exist
functions $z\in W^{1,q}(\Omega,\cR^3)$ and $z_1\in L^q(\Omega,W^{1,q}_{per}(Y,\cR^3))$
such that
$${\cal T}_\eta(\na z_\eta)\rightharpoonup \na{z}+\na_yz_1 \ \ {\rm in}\ 
L^{q}(\Omega\times Y,{\cal M}^3).$$
Passage to the weak limit in (\ref{EstimateUnfoldingCurl3}) yields now that
$$v_0=h+\Curl w+\na z+\na_y z_1,$$
where on the right hand side the function $z_1$ depends on the variable $y$ only. 
Therefore, we get that
$$\Curl_yv_0=\Curl_y\na_y z_1=0.$$
The proof of Theorem~\ref{HelmholtzZerlegungThMain} is complete.
\end{proof}




\section{Homogenized system of equations}\label{Homogenization}

\paragraph{Main result.} First, we define a class of maximal monotone functions we deal with
in this work.
\begin{Def}\label{CoercClass}
For $m\in L^1(\Omega, \cR)$, $\alpha\in{\mathbb R}_{+}$ and $q>1$, 
${\mathbb M}(\Omega,{\mathbb R}^k,q,\alpha,m)$ is the set of 
multi-valued functions $h:\Omega \times {\mathbb R}^k\rightarrow 2^{\cR^k}$
with the following properties
\begin{itemize}
    \item $v \mapsto h(x, v)$ is maximal monotone for almost all $x \in \Omega$,
    \item the mapping $x \mapsto j_{\lambda}(x, v) : \Omega \rightarrow \cR^k$ is
measurable for all $\lambda > 0$, where $j_{\lambda}(x, v)$ is
     the inverse of $v \mapsto v + \lambda h(x, v)$,
    \item for a.e. $x\in \Omega$ and every $v^*\in h(x,v)$
     \begin{eqnarray}
\label{inequMain} \alpha\left(\frac{|v|^q}{q}+\frac{|v^*|^{q^*}}{q^*}\right)
\le (v,v^*)+m(x),
\end{eqnarray}
where $1/q+1/{q^*}=1$.
\end{itemize}
\end{Def}
\begin{Rem}{\rm
We note that the condition \eq{inequMain} is equivalent to the following
two inequalities
\begin{eqnarray}
\label{inequMain1} &&|v^*|^{q^*}\le m_1(x)+\alpha_1|v|^q,\\
&& (v,v^*)\ge m_2(x)+\alpha_2|v|^q,\label{inequMain2}
\end{eqnarray}
for a.e. $x\in \Omega$ and every $v^*\in h(x,v)$ and with suitable functions
$m_1,m_2\in L^1(\Omega, \cR)$ and numbers 
$\alpha_1,\alpha_2\in{\mathbb R}_{+}$.} 
\end{Rem}
\begin{Rem} Visco-plasticity is typically included in the former conditions by choosing the function $g$ to be in Norton-Hoff form, i.e. 
\begin{align*}
  g(\Sigma)=[|{\Sigma}|-\yieldlimit]_+^r\,\frac{\Sigma}{|\Sigma|}\, , \quad\Sigma\in{\cal M}^3\, ,
\end{align*}  
where $\yieldlimit$ is the flow stress and $r$ is some parameter together with $[x]_+:=\max(x,0)$. If $g:{\cal M}^3\mapsto {\cal S}^3$ then the flow is called irrotational (no plastic spin).
 \end{Rem}
The main properties of the class ${\mathbb M}(\Omega,{\mathbb R}^k,q,\alpha,m)$ are
collected in the following proposition.
\begin{Prop}\label{MainClassMaxMonoProp}
Let ${\cal H}$ be a canonical extension of a function $h:{\mathbb R}^k\to 2^{{\mathbb R}^k}$,
which belongs to
    ${\mathbb M}(\Omega,{\mathbb R}^k,q,\alpha,m)$. Then ${\cal H}$
  is maximal monotone, surjective and $D({\cal H})=L^p(\Omega,{\mathbb R}^k)$.
     \end{Prop}
\begin{proof}
 See Corollary 2.15  in \cite{Damlamian07}. 
\end{proof}
In linear elasticity theory it is well known 
(see \cite[Theorem 4.2]{Vale88}) 
that a Dirichlet boundary value problem formed by the equations
\begin{eqnarray}
\label{elast1} -\di_x \sigma_\eta(x) &=& \hat{b}(x), \hskip4.8cm \ x \in \Omega,\\
\label{elast2} \sigma_\eta(x) &=&  {\mathbb C}[x/\eta]( \sym
 {({ \nabla _x u_\eta(x)})}-{\hat\varepsilon}_\eta(x)),  \quad  \quad x \in \Omega,\\
 \label{elast3} u_\eta(x) &=&0, \hskip5.2cm  \ x \in {
\partial \Omega},
\end{eqnarray}
to given $\hat{b} \in H^{-1}( \Omega, {\mathbb R}^3)$ and
${\hat \varepsilon}_\eta \in L^2( \Omega, {\cal S}^3)$ has a unique
weak solution $(u_\eta, \sigma_\eta) \in  H^{1}_0( \Omega, {\mathbb R}^3)
 \times L^2 ( \Omega, {\cal S}^3)$.
 
Next, we define
the notion of strong solutions for 
the initial boundary value problem (\ref{CurlPr1}) - (\ref{CurlPr6}).
\begin{Def}\label{StrongSol}{\bf (Strong solutions)}
 A function $(u_\eta,\sigma_\eta,p_\eta)$ such that
\[(u_\eta,\sigma_\eta)\in H^{1}(0,T_e; H^{1}_0(\Omega, {\mathbb R}^3) \times
 L^{2} ({\Omega}, {\cal S}^3)),\ \ 
\Sigma^{\rm lin}_\eta\in L^{q}(\Omega_{T_e}, {\cal M}^3),\] \[ p_\eta\in   
H^{1}(0,T_e;L^{2}_{\Curl}(\Omega,{\cal M}^3))\cap 
L^{2}(0,{T_e};Z^{2}_{\Curl}(\Omega,{\cal M}^3))\]
is called a {\rm strong solution} of the initial boundary value
problem (\ref{CurlPr1}) - (\ref{CurlPr6}), if for every $t\in [0,T_e]$
the function $(u_\eta(t),\sigma_\eta(t))$ is a weak solution of the boundary value problem
 (\ref{elast1}) - (\ref{elast3}) with ${\hat\varepsilon}_p=\sym p_\eta(t)$ and
$\hat{b}=b(t)$, the evolution inclusion (\ref{CurlPr3}) and the initial condition (\ref{CurlPr4}) are satisfied pointwise.
\end{Def}
Next, we state the existence result (see \cite{NesenenkoNeff2012}).
\begin{Theor}\label{existMain} 
Suppose that $1<q^*\le 2\le q<\infty$. Assume that $\Omega \subset \cR^3$
 is a sliceable domain with a $C^2$-boundary,  $C_1\in L^\infty(\Omega, \mathbb{R})$ and
 ${\mathbb C}\in L^\infty(\Omega,{\cal S}^3)$ satisfying (\ref{AssOnC1}) and
(\ref{AssOnC}), respectively. Let the functions  
$b \in W^{1,q}(0,T_e; L^{q}(\Omega, {\mathbb R}^3))$ be given and
 $g\in{\mathbb M}(\Omega, {\cal M}^3, q, \alpha, m)$.
 Suppose that for a.e. $x\in\Omega$ the relation
 \begin{eqnarray}
\label{Assumption1}  0\in g(x/\eta, \sigma^{(0)}(x))
\end{eqnarray}
holds, where the function 
$\sigma^{(0)}\in L^2({\Omega}, {\cal S}^3)$ is  determined by equations
(\ref{elast1}) - (\ref{elast3}) for ${\hat\varepsilon}_p=0$ and $\hat{b}=b(0)$. 
 Then there exists a strong unique solution $(u_\eta,\sigma_\eta,p_\eta)$
of the initial boundary value problem (\ref{CurlPr1}) - (\ref{CurlPr6}).
\end{Theor} 
Now we can formulate the main result of this work.
\begin{Theor}\label{HomoMain} 
Suppose that all assumptions of Theorem~\ref{existMain} are fulfilled.  Then there exists
\begin{eqnarray}
&&u_0\in  H^{1}(0,T_e;H^{1}_0(\Omega,{\mathbb R}^3)),\ \ 
u_1\in  H^{1}(0,T_e;L^{2}(\Omega, H^{1}_{per}(Y, {\mathbb R}^3))),\non\\[1ex]
&&\sigma_0\in  L^{\infty}(0,T_e;L^2(\Omega\times Y,{\cal S}^3)),\ \
\sigma\in  L^{\infty}(0,T_e;L^2(\Omega,{\cal S}^3)), \non\\[1ex]
&&
p\in H^{1}(0,T_e;L^{2}(\Omega,{\cal M}^3))\cap 
L^{2}(0,{T_e};Z^{2}_{\Curl}(\Omega,{\cal M}^3)),\non\\[1ex]
 && p_0\in H^{1}(0,T_e;L^{2}(\Omega\times Y, {\cal M}^3))
  \  with \ \Curl_y p_0=0,\non
\end{eqnarray}
and $$p_1\in L^{2}(0,T_e;L^{2}(\Omega, W^{2,q^*}_{per}(Y, {\cal M}^3)))
 \  with \ \di_y p_1=0,$$
 such that
\begin{eqnarray}
&&u_\eta\rightharpoonup u_0 \ \ in \ H^{1}(0,T_e;H^{1}_0(\Omega,\cR^3)),
\label{HomEq1}\\[1ex]
&&p_\eta\rightharpoonup p \ \ in \ H^{1}(0,T_e;L^{2}(\Omega,{\cal M}^3))\cap 
L^{2}(0,{T_e};Z^{2}_{\Curl}(\Omega,{\cal M}^3)),
\label{HomEq11}\\[1ex]
&&{\cal T}_\eta(\na u_\eta)\rightharpoonup \na{u_0}+\na_yu_1 \ \ in 
\ H^{1}(0,T_e;L^{2}(\Omega\times Y,\cR^3)),\label{HomEq2}\\[1ex]
&&\sigma_\eta\overset{*}{\rightharpoonup} \sigma \ \ in \ 
L^{\infty}(0,T_e;L^2(\Omega,{\cal S}^3)),\label{HomEq3a}\\[1ex]
&&{\cal T}_\eta(\sigma_\eta)\overset{*}{\rightharpoonup} \sigma_0 \ \ in \ 
L^{\infty}(0,T_e;L^2(\Omega\times Y,{\cal S}^3)),\label{HomEq3}\\[1ex]
&&{\cal T}_\eta(p_\eta)\rightharpoonup p_0\ \ in \ 
L^{2}(0,T_e;L^{2}(\Omega\times Y,{\cal M}^3)),\label{HomEq4}\\[1ex]
&&{\cal T}_\eta(\partial_t p_\eta)\rightharpoonup \partial_t p_0 \ \ in \ 
L^{2}(0,T_e;L^{2}(\Omega\times Y,{\cal M}^3)),\label{HomEq41}
\end{eqnarray}
and
\begin{eqnarray}
&&{\cal T}_\eta(\Curl p_\eta)\rightharpoonup \Curl p\ \ in \ 
L^{2}(0,T_e;L^{2}(\Omega, H^1_{per}(Y,{\cal M}^3))),\label{HomEq42}\\[1ex]
&&{\cal T}_\eta(\dev\sym p_\eta)\rightharpoonup \dev\sym p_0\ \ in \ 
L^{2}(\Omega_{T_e}\times Y,{\cal M}^3),\label{HomEq43}\\[1ex]
&&{\cal T}_\eta(\Curl\Curl p_\eta)\rightharpoonup \tilde{p} \ \ in \ 
L^{2}(\Omega_{T_e}\times Y,{\cal M}^3),\label{HomEq44}\\[1ex]
&&{\cal T}_\eta(\Sigma^{\rm lin}_\eta)\rightharpoonup \Sigma^{\rm lin}_0 \ \ in \ 
L^{q}(\Omega_{T_e}\times Y,{\cal M}^3),\label{HomEq45}
\end{eqnarray}
where 
\begin{eqnarray}
&&\tilde{p}:=\Curl\Curl p+\Curl_y\Curl_y p_1,\non\\
&&\Sigma^{\rm lin}_0:=\sigma_0-C_1[y] \dev \sym p_0- C_2\tilde p,\non
\end{eqnarray}
and $(u_0,u_1,\sigma, \sigma_0,p, p_0, p_1)$ is a solution of the
following system of equations:
\begin{eqnarray}
\label{HomogEqua1}  -\di_x \sigma(x,t)  &=& b(x, t), \\
\label{HomogEqua2}  -{\rm div}_y \sigma_0(x, y, t) &=& 0, \\[1ex]
  \sigma_0(x, y, t) &=& {\mathbb C}[y] ( \sym (\na_x u_0(x,t)+\na_y u_1(x,y,t) - p_0(x,y,t))), \label{HomogEqua3}   \\[1ex]
\label{HomogEqua4}
\partial_t p_0(x,y,t)&\in& g(y, \Sigma^{\rm lin}_0(x,y,t)),
\end{eqnarray}
which holds for $(x, y, t) \in \Omega \times \cR^3 \times [0,T_e ]$, and the initial
condition and boundary condition 
\begin{eqnarray} 
p_0(x,y,0) &=& 0, \hspace{7ex} x \in \Omega,  \label{HomCurlPr4}   \\
p(x,t)\times n (x)&=&0,\hspace{7ex} (x,t) \in \partial\Omega \times
  [0,T_e), \label{HomCurlPr5}\\
u_0(x,t) &=& 0, \quad\quad\quad (x,t) \in \partial \Omega \times
  [0,T_e).\,,\label{HomCurlPr6}
\end{eqnarray}
The functions $\sigma$  and  
$p$  are related to
$\sigma_0$ and $p_0$ in the following ways
 \[\sigma(x, t)=\int_{Y}\sigma_0(x,y,t) dy, \ \ p(x, t)=\int_{Y}p_0(x,y,t) dy.\]
\end{Theor} 
The proof of Theorem~\ref{HomoMain} is divided into two parts. In the next lemma
we derive the uniform estimates for $(u_\eta,\sigma_\eta, p_\eta)$ and then,
based on these estimates, we show the convergence result.
\subsection{Uniform estimates} 
First, we show that the sequence of solutions $(u_\eta,\sigma_\eta, p_\eta)$ is weakly compact.
\begin{Lem}\label{existLemma} Let all assumptions of Theorem~\ref{HomoMain} 
be satisfied.
Then the sequence
of solutions $(u_\eta,\sigma_\eta)$ is weakly compact in 
$H^{1}(0,T_e; H^{1}_0(\Omega, {\mathbb R}^3) \times
 L^{2} ({\Omega}, {\cal S}^3))$ and  $p_\eta$ is weakly compact in
 $H^{1}(0,T_e;L^{2}(\Omega,{\cal M}^3))\cap 
L^{2}(0,{T_e}, Z^{2}_{\Curl}(\Omega,{\cal M}^3))$. 
\end{Lem} 
\begin{proof} 
To prove the lemma we recall the basic steps in the proof of the existence result
(Theorem~\ref{existMain}). For more details the reader is referred to 
\cite{NesenenkoNeff2012}. The time-discretized problem for (\ref{CurlPr1}) - (\ref{CurlPr6}) is introduced as follows:\\
Let us fix any
$m\in{\mathbb N}$ and set
\[h:=\frac{T_e}{2^m}, \ p^0_{\eta,m}:=0\ b^n_m:=\frac{1}{h}\int^{nh}_{(n-1)h}b(s)ds\in 
L^{q}( \Omega, {\mathbb R}^3),\ \ n=1,...,2^m. \]
Then we are looking for functions $u^n_{\eta,m}\in H^1(\Omega,{\mathbb R}^3)$,
$\sigma^n_{\eta,m}\in L^2(\Omega,{\cal S}^3)$ and $p^n_{\eta,m}\in
 Z^2_{\Curl}(\Omega,{\cal M}^3)$ with $p^n_{\eta,m}(x)\in\mathfrak{sl}(3)$ 
 for a.e. $x\in\Omega$ and
\[\Sigma^{\rm lin}_{n,m}:=\sigma^n_{\eta,m}-C_1[x/\eta] \dev \sym p^n_{\eta,m} -
\frac{1}{m}p^n_{\eta,m}-C_2\Curl\Curl p^n_{\eta,m}\in L^q(\Omega,{\cal M}^3)\]
solving the following problem
\begin{eqnarray}
- \di_x \sigma^n_{\eta,m}(x) &=&  b^n_m(x), \label{CurlPr1Dis}
\\[1ex]
\sigma^n_{\eta,m}(x) &=& 
{\mathbb C} [x/\eta]( \sym (\na_x u^n_{\eta,m}(x) - p^n_{\eta,m}(x) ) )   \label{CurlPr2Dis}
\\[1ex] 
\label{CurlPr3Dis} \frac{p^n_{\eta,m}(x)-p^{n-1}_{\eta,m}(x)}{h} & \in & 
g \big(x/\eta,\Sigma^{\rm lin}_{n,m}(x)\big),\label{microPr3Dis} 
\end{eqnarray}
together with the boundary conditions 
\begin{eqnarray} 
p^n_{\eta,m}(x)\times n (x)&=&0,\hspace{7ex} x \in \partial\Omega, 
\label{CurlPr5Dis}\\
u^n_{\eta,m}(x) &=& 0, \quad\quad\quad x \in \partial \Omega\,.\label{CurlPr6Dis}
\end{eqnarray}
Such functions $(u^n_{\eta,m}, \sigma^n_{\eta,m}, p^n_{\eta,m})$ exist 
 and satisfy the following estimate
\begin{eqnarray}\label{aprioriEstim1}
&&\frac{1}2\Big(
\|{\mathbb B}^{1/2}\sigma^l_{\eta,m}\|^2_2+\alpha_1\|\dev \sym p^l_{\eta,m}\|^2_2
+\frac{1}{m}\|p^l_{\eta,m}\|^2_2+ C_2\|\Curl p^l_{\eta,m}\|^2_2\Big)\non\\
&&+h\hat{C}\sum^l_{n=1}\left(\left\| \Sigma^{\rm lin}_{n,m}\right\|^{q}_{q}+
\Big\|\frac{p^n_{\eta,m}-p^{n-1}_{\eta,m}}h\Big\|^{q^*}_{q^*}\right)
\le
C^{(0)}+\int_\Omega m(x)dx
\\&&\hspace{3ex}
+h\tilde{C}\sum^l_{n=1}\Big(\|b^n_m\|^q_{q}+
\|(b^n_m-b^{n-1}_m)/h\|^{2}_{2}\Big)\non
\end{eqnarray}
for any fixed $l\in[1,2^m]$, where (here ${\mathbb B}:={\mathbb C}^{-1}$)
\[2C^{(0)}:=
\|{\mathbb B}^{1/2}\sigma^{(0)}\|^2_2\]
and  $\tilde{C}$, $\hat{C}$ are some positive
constants independent of $\eta$ (see \cite{NesenenkoNeff2012} for details). To proceed further
we introduce the Rothe approximation functions.\vspace{1ex}\\
{\bf Rothe approximation functions:} 
For any family $\{\xi^{n}_m\}_{n=0,...,2m}$ of functions in a reflexive Banach
space $X$, we define
{\it the piecewise affine interpolant} $\xi_m\in C([0,T_e],X)$ by
\begin{eqnarray}\label{RotheAffineinterpolant}
\xi_m(t):= \left(\frac{t}h-(n-1)\right)\xi^{n}_m+
\left(n-\frac{t}h\right)\xi^{n-1}_m \ \ {\rm for} \ (n-1)h\le t\le nh
\end{eqnarray}
and {\it the piecewise constant interpolant} $\bar\xi_m\in L^\infty(0,T_e;X)$ by
\begin{eqnarray}\label{RotheConstantinterpolant}
\bar\xi_m(t):=\xi^{n}_m\  {\rm for} \ (n-1)h< t\le nh, \ 
 n=1,...,2^m, \ {\rm and} \ \bar\xi_m(0):=\xi^{0}_m.
\end{eqnarray}
For the further analysis we recall the following property of $\bar\xi_m$
and $\xi_m$: 
\begin{eqnarray}\label{RotheEstim}
\|\xi_m\|_{L^q(0,T_e;X)}\le\|\bar\xi_m\|_{L^q(-h,T_e;X)}\le
\left(h \|\xi^{0}_m\|^q_X+\|\bar\xi_m\|^q_{L^q(0,T_e;X)}\right)^{1/q},
\end{eqnarray}
where $\bar\xi_m$ is formally extended to $t\le0$ by $\xi^{0}_m$ and
$1\le q\le\infty$ (see \cite{Roubi05}).\vspace{1.5ex}\\
Now, from (\ref{aprioriEstim1}) we get immediately that
\begin{eqnarray}
&&
\bar{C}\|\bar\sigma_{\eta,m}(t)\|^2_{\Omega}+\alpha_1\|\dev \sym\bar p_{\eta,m}(t)\|^2_2
+\frac{1}{m}\|\bar p_{\eta,m}(t)\|^2_2+ C_2\|\Curl\bar p_{\eta,m}(t)\|^2_2\non\\
&&\hspace{3ex}+
2\hat{C}\left(\|\partial_t p_{\eta,m}\|^{q^*}_{q^*,\Omega\times(0, T_e)}+
\|\bar\Sigma^{\rm lin}_{m}\|^{q}_{q,\Omega\times(0, T_e)}\right)\label{aprioriEstim2A}
\\
&&\hspace{3ex}\le 2C^{(0)}+2\|m\|_{1,\Omega}
+2\tilde{C}\|b\|^q_{W^{1,q}(0,{T_e};L^q(\Omega,{\cal S}^3))},\non
\end{eqnarray}
where $\bar{C}$ is some other constant independent of $\eta$.  In \cite{NesenenkoNeff2012}
it is shown that the Rothe approximation functions 
$(u_{\eta,m}, \sigma_{\eta,m}, p_{\eta,m})$ and 
$(\bar u_{\eta,m}, \bar\sigma_{\eta,m}, \bar p_{\eta,m})$ converge to the same limit
$(u_{\eta}, \sigma_{\eta}, p_{\eta})$. Due to the lower semi-continuity of the
norm and (\ref{aprioriEstim2A}) this convergence is uniform with respect to $\eta$.
Therefore, estimate (\ref{aprioriEstim2A}) provides that
\begin{eqnarray}
&&\{\sigma_\eta\}_\eta \ {\rm is}\  {\rm uniformly}\  {\rm bounded}\  {\rm in}
 \ L^{\infty}(0,T_e;L^2(\Omega,{\cal S}^3)),\label{aprioriEstim2}\\[1ex]
 &&\{\dev\sym p_\eta\}_\eta \ {\rm is}\  {\rm uniformly}\  {\rm bounded}\  {\rm in}
 \ L^\infty(0,{T_e};L^{2}(\Omega,{\cal M}^3)),\label{aprioriEstim3new1}\\[1ex]
&&\{\Curl p_\eta\}_\eta\ {\rm is}\  {\rm uniformly}\  {\rm bounded}\  {\rm in}
 \ L^\infty(0,{T_e};L^{2}(\Omega,{\cal M}^3)),\label{aprioriEstim3new2}\\[1ex]
&&\{p_\eta\}_\eta  \ {\rm is}\  {\rm uniformly}\  {\rm bounded}\  {\rm in}
 \ W^{1,{q^*}}(0,{T_e};L^{q^*}(\Omega,{\cal M}^3)),\label{aprioriEstim3}\\[1ex]
&&\{\Sigma^{\rm lin}_\eta\}_\eta \ {\rm is}\  {\rm uniformly}\  {\rm bounded}\  {\rm in}
 \ L^{q}(\Omega_{T_e},{\cal M}^3).\label{aprioriEstim6}
\end{eqnarray}
Furthermore, from estimates (\ref{incompatible_korn}),
 (\ref{aprioriEstim2}) - (\ref{aprioriEstim6}) we obtain
easily that
\begin{eqnarray}
&&\{u_\eta\}_\eta \ {\rm is}\  {\rm uniformly}\  {\rm bounded}\  {\rm in}
 \ L^{2}(0,T_e;H^{1}_0(\Omega,{\mathbb R}^3)),\label{aprioriEstim7}\\[1ex]
&&\{p_\eta\}_\eta \ {\rm is}\  {\rm uniformly}\  {\rm bounded}\  {\rm in}\
L^{2}(0,{T_e};Z^{2}_{\Curl}(\Omega,{\cal M}^3)).\label{aprioriEstim9}
\end{eqnarray}
{\bf Additional regularity of discrete solutions.}  
In order to get the additional 
a'priori estimates, we extend the function $b$ to $t<0$ by setting
$b(t)=b(0)$. 
The extended function $b$ is in the space
$W^{1,p}(-2h,T_e;W^{-1,p}(\Omega,{\mathbb R}^3))$. Then, we set
 $b^0_m=b^{-1}_m:=b(0)$. 
Let us further set $$p_{\eta,m}^{-1}:=p_{\eta,m}^0-h {\cal G}_\eta(\Sigma^{\rm lin}_{0,m}),$$
where ${\cal G}_\eta:L^p(\Omega, {\cal M}^3)\to2^{L^q(\Omega,\mathfrak{sl}(3))}$ denotes
 the canonical extensions of $g(x/\eta,\cdot):{{\cal M}^3}\to2^{\mathfrak{sl}(3)}$.
The assumption (\ref{Assumption1}) and the homogeneous initial condition
 imply that $p_{\eta,m}^{-1}=0$.
Next, we define functions
$(u_{\eta,m}^{-1}, \sigma_{\eta,m}^{-1})$ and 
$(u_{\eta,m}^0, \sigma_{\eta,m}^0)$ as solutions of the 
linear elasticity problem (\ref{elast1}) - (\ref{elast3})
to the data $\hat b=b^{-1}_m$, $\hat\gamma=0$, $\hat\ve_p=0$ 
and $\hat b=b^0_m$, $\hat\gamma=0$, $\hat\ve_p=0$,
respectively. Obviously, 
the following estimate holds
\begin{eqnarray}\label{IneqPrepare}
\left\{
\left\|\frac{u_{\eta,m}^0-u_{\eta,m}^{-1}}h\right\|_{2},
\left\|\frac{\sigma_{\eta,m}^0-\sigma_{\eta,m}^{-1}}h\right\|_{2}\right\}\le C,
\end{eqnarray}
where $C$ is some positive constant independent of $m$ and $\eta$.
Taking now the incremental ratio of (\ref{CurlPr3Dis}) for
$n=1,...,2^m$, we obtain\footnote{For sake of
simplicity we use the following notation 
$\rt\phi^n_m:=(\phi^n_m-\phi^{n-1}_m)/h$, where 
 $\phi^{0}_m,\phi^{1}_m,...,\phi^{2m}_m$ is any family of functions.}
\[\rt p^n_{\eta,m}-\rt p^{n-1}_{\eta,m}={\cal G}_\eta(\Sigma ^{\rm lin}_{n,m})-
{\cal G}_\eta(\Sigma ^{\rm lin}_{(n-1),m}).\]
Let us now multiply the last identity by 
$-(\Sigma^{\rm lin}_{n,m}- \Sigma^{\rm lin}_{(n-1),m})/h$.
Then using the monotonicity of ${\cal G}_\eta$ we obtain that
\[\frac{1}{m} \Big(\rt p^n_{\eta,m}-\rt p^{n-1}_{\eta,m}, \rt p^n_{\eta,m}\Big)_\Omega
+\big(\rt p^n_{\eta,m}-\rt p^{n-1}_{\eta,m}, C_1\dev \sym(\rt p^n_{\eta,m})\big)_\Omega\]\[
+\big(\rt p^n_{\eta,m}-\rt p^{n-1}_{\eta,m}, C_2\Curl\Curl(\rt p^n_{\eta,m})\big)_\Omega
\le\big(\rt p^n_{\eta,m}-\rt p^{n-1}_{\eta,m},\rt\sigma^{n}_{\eta,m}\big)_\Omega.
\]
 With 
(\ref{CurlPr1Dis}) and (\ref{CurlPr2Dis}) the previus inequality can be
rewritten as follows
\[\frac{1}{m} \Big(\rt p^n_{\eta,m}-\rt p^{n-1}_{\eta,m}, \rt p^n_{\eta,m}\Big)_\Omega
+\big(\rt p^n_{\eta,m}-\rt p^{n-1}_{\eta,m}, C_1\dev \sym(\rt p^n_{\eta,m})\big)_\Omega\]\[
+\big(\rt p^n_{\eta,m}-\rt p^{n-1}_{\eta,m}, C_2\Curl\Curl(\rt p^n_{\eta,m})\big)_\Omega
+\big(\rt \sigma^n_{\eta,m}-\rt \sigma^{n-1}_{\eta,m},{\mathbb C}^{-1}\rt\sigma^{n}_{\eta,m}\big)_\Omega
\]\[\le
\big(\rt u^n_{\eta,m}-\rt u^{n-1}_{\eta,m},\rt b^{n}_m\big)_\Omega.
\]
As in the proof of (\ref{aprioriEstim1}),
 multiplying  the last inequality by $h$ and summing with respect to $n$
 from 1 to $l$ 
for any fixed $l\in[1,2^m]$ we get the estimate
\begin{eqnarray}\label{IneqExstSols}
&&\frac{h}{m}\|\rt p^l_{\eta,m}\|^2_{2}+h\alpha_1\|\dev\sym\rt p^l_{\eta,m}\|^2_{2}+
h\|{\mathbb B}^{1/2}\rt\sigma^l_{\eta,m}\|^2_{2}\non\\[1ex]
&&+hC_2
\|\Curl\rt p^l_{\eta,m}\|^2_{2}\le 2hC^{(0)}
+ 2h\sum^l_{n=1}\big(\rt u^n_{\eta,m}-\rt u^{n-1}_{\eta,m},\rt b^{n}_m\big)_\Omega,
\end{eqnarray}
where now $C^{(0)}$ denotes
\[2C^{(0)}:=\|{\mathbb B}^{1/2}\rt \sigma^0_{\eta,m}\|^2_{2}.\]
We note that \eq{IneqPrepare} yields the uniform
boundness of $C^{(0)}$ with respect to $m$. Now, using Young's
inequality with $\epsilon>0$ in (\ref{IneqExstSols}) and then
summing the resulting inequality for $l=1,...,2^m$ we derive the inequality
\begin{eqnarray}\label{aprioriEstim21}
\frac{1}{m}\left\|\partial_tp_m\right\|^2_{2,\Omega_{T_e}}+\alpha_1
\left\|\dev \sym\left(\partial_tp_m\right)\right\|^2_{2,\Omega_{T_e}}+
C_2\left\|\Curl\left(\partial_tp_m\right)\right\|^2_{2,\Omega_{T_e}}\\
+C\left\|\partial_t\sigma_m\right\|^2_{2,\Omega_{T_e}}\le 
C_\varepsilon\|\partial_tb_m\|^2_{2,\Omega_{T_e}}+2\varepsilon
\|\partial_tu_m\|^2_{2,\Omega_{T_e}},\non
\end{eqnarray}
where $C_\varepsilon$ is some positive constant independent of $m$ and $\eta$. 
Using now inequality (\ref{incompatible_korn}), the condition 
$\partial_t p_m(x, t)\in\mathfrak{sl}(3)$ for a.e. $(x, t)\in\Omega_{T_e}$,  and the 
ellipticity theory of linear systems we obtain that
\begin{eqnarray}\label{aprioriEstim22}
\frac{1}{m}\left\|\partial_tp_m\right\|^2_{2,\Omega_{T_e}}+
C_\epsilon(\Omega)\left\|\partial_tp_m\right\|^2_{2,\Omega_{T_e}}+
C\left\|\partial_t\sigma_m\right\|^2_{2,\Omega_{T_e}}\le 
C_\varepsilon\|\partial_tb_m\|^2_{2,\Omega_{T_e}},
\end{eqnarray}
where $C_\varepsilon(\Omega)$ is some further positive constant independent of $m$ and $\eta$.
Since $b_m$ is uniformly bounded in $W^{1,q}(\Omega_{T_e}, {\cal S}^3)$, 
estimates (\ref{aprioriEstim21}) and (\ref{aprioriEstim22}) imply 
\begin{eqnarray}
&&\{\dev\sym\partial_t p_\eta\}_\eta \ {\rm is}\  {\rm uniformly}\  {\rm bounded}\  {\rm in}
 \ L^2(0,{T_e};L^{2}(\Omega,{\cal M}^3)),\label{aprioriEstim7}\\[1ex]
&&\{\partial_t\sigma_\eta\}_\eta\ {\rm is}\  {\rm uniformly}\  {\rm bounded}\  {\rm in}
 \ L^2(0,{T_e};L^{2}(\Omega,{\cal M}^3)),\label{aprioriEstim7a}\\[1ex]
&&\{\Curl\partial_t p_\eta\}_\eta\ {\rm is}\  {\rm uniformly}\  {\rm bounded}\  {\rm in}
 \ L^2(0,{T_e};L^{2}(\Omega,{\cal M}^3)),\label{aprioriEstim7aa}\\[1ex]
&&\{p_\eta\}_\eta \ {\rm is}\  {\rm uniformly}\  {\rm bounded}\  {\rm in}
 \ H^1(0,{T_e};L^{2}_{\Curl}(\Omega,{\cal M}^3)). \label{aprioriEstim9}
\end{eqnarray}
The proof of the lemma is complete.
\end{proof}
\subsection{Proof of Theorem~\ref{HomoMain}} 
Now, we can prove Theorem~\ref{HomoMain}.
\begin{proof} Due to Lemma~\ref{existLemma}, we have that  the sequence
of solutions $(u_\eta,\sigma_\eta)$ is weakly compact in 
$H^{1}(0,T_e; H^{1}_0(\Omega, {\mathbb R}^3) \times
 L^{2} ({\Omega}, {\cal S}^3))$ and the sequence
$p_\eta$ is weakly compact in
 $H^{1}(0,T_e;L^{2}(\Omega,{\cal M}^3))\cap 
L^{2}(0,{T_e};Z^{2}_{\Curl}(\Omega,{\cal M}^3))$. 
Thus, by Proposition~\ref{UnfoldWeakConv}, Proposition~\ref{UnfoldGradient} and
Theorem~\ref{HelmholtzZerlegungThMain}, the uniform estimates (\ref{aprioriEstim2}) -
(\ref{aprioriEstim9}) yield that
there exist functions $u_0$, $u_1$, $\sigma$, $\sigma_0$, $p$, $p_0$ and $p_1$ with the prescribed
regularities in Theorem~\ref{HomoMain} such that the convergences in
\eq{HomEq1} - \eq{HomEq45} hold.  Note that 
\eq{HomEq2} - \eq{HomEq4} give the equation \eq{HomogEqua3}, i.e
\begin{eqnarray}\label{HelpEq0}
\sigma_0(x, y, t) = {\mathbb C}[y] \big( \sym (\na_x u_0(x, t)+\na_y u_1(x, y, t)
  - p_0(x, y, t))\big), \ \ {\rm a.e.}
\end{eqnarray}
By Proposition~\ref{UnfoldWeakConv}, the weak-star limit $\sigma$ of $\sigma_\eta$ in 
$L^{\infty}(0,T_e;L^2(\Omega,{\cal S}^3))$ and the weak limit  
$p$ of $p_\eta$ in $L^{2}(0,T_e;L^{2}(\Omega,{\cal M}^3))$ are related to
$\sigma_0$ and $p_0$ in the following ways
 \[\sigma(x, t)=\int_{Y}\sigma_0(x,y,t) dy, \ \ p(x, t)=\int_{Y}p_0(x,y,t) dy.\]
Now, as in \cite{Damlamian09a}, we consider any
 $\phi\in C_0^\infty(\Omega,{\mathbb R}^3)$. Then, 
by the weak convergence of $\sigma_\eta$,
the passage to the weak limit in \eq{CurlPr1} yields
\begin{eqnarray}\label{HelpEq1}
\int_\Omega(\sigma(x,t),\na\phi(x))dx=\int_\Omega(b(x,t),\phi(x))dx,
\end{eqnarray}
i.e $\di_x\sigma=b$ in the sense of distributions. 
Next, define 
$\phi_\eta(x)=\eta\phi(x)\psi(x/\eta)$, where 
$\phi\in C_0^\infty(\Omega,{\mathbb R}^3)$
and $\psi\in C_{per}^\infty(Y,{\mathbb R}^3)$. Then, one obtains
 that 
$$\phi_\eta\rightharpoonup 0,\ \ {\rm in}\ H^{1}_{0}(\Omega,{\mathbb R}^3),
\ \ {\rm and}\ \ 
{\cal T}_\eta(\na \phi_\eta)\to \phi\na_y\psi ,\ \ {\rm in}\ 
L^2(\Omega,H^{1}_{per}(Y,{\mathbb R}^3)).$$ 
Therefore, since $\phi_\eta$ has a compact support,
\begin{eqnarray}\label{HelpIntId}
\int_{\Omega\times Y}({\cal T}_\eta(\sigma_\eta(t)),
{\cal T}_\eta(\na\phi_\eta))dxdy=
\int_{\Omega}(b(t),\phi_\eta)dx. 
\end{eqnarray} 
The passage to the limit in \eq{HelpIntId} leads to
 \[\int_{\Omega\times Y}
(\sigma_0(x,y,t),\phi(x)\na_y\psi(y))dxdy=0.\]
Thus, in virtue of the arbitrariness of $\phi$, one can conclude that
\begin{eqnarray}\label{HelpEq2}
\int_{\Omega\times Y}
(\sigma_0(x,y,t),\na_y\psi(y))dxdy=0.
\end{eqnarray} 
i.e $\di_y \sigma_0(x,\cdot,t)=0$ in the sense of distributions.

Next, let 
${\cal T}_\eta({\cal G}_\eta):L^p(\Omega\times Y,{\mathbb R^N})\to2^{L^q(\Omega\times Y,{\mathbb R^N})}$
 and ${\cal G}:L^p(\Omega,{\mathbb R^N})\to2^{L^q(\Omega,{\mathbb R^N})}$ denote the canonical extensions of ${\cal T}_\eta(g_\eta)(x,y):{\mathbb R^N}\to2^{\mathbb R^N}$
and $g(y):{\mathbb R^N}\to2^{\mathbb R^N}$, respectively. Here, $g(y)$ is the pointwise  limit
graph of the convergent sequence of graphs ${\cal T}_\eta(g_\eta)(x,y)$.  The existence
of the limit graph for ${\cal T}_\eta(g_\eta)(x,y)$ guaranteed by Theorem~\ref{convMaxMonGrEquiv}.
Indeed, the resolvent $j_\lambda^{{\cal T}_\eta(g_\eta)}$ converges pointwise to the
resolvent $j_\lambda^{g}$, what follows from the periodicity of the mapping
$y \rightarrow g(y,z):Y\to2^{\mathbb R^N}$ and the simple computations:
\[j_\lambda^{{\cal T}_\eta(g_\eta)}(x,y,z)={\cal T}_\eta(j_\lambda^{g_\eta})(x,y,z)
=j_\lambda^{g}(y,z),
\]
for a.e. $(x,y)\in\Omega\times Y$ and every $z\in{\mathbb R}^N$. Thus, by 
Theorem~\ref{convMaxMonGrEquiv} we get that
\begin{eqnarray}\label{PointwiseGraph}
{\cal T}_\eta(g_\eta)(x,y)\rightarrowtail g(y)
\end{eqnarray}
holds for a.e. $(x,y)\in\Omega\times Y$. Since $g_\eta\in {\cal M}(\Omega,{\mathbb R}^N,p,\alpha,m)$, by Definition~\ref{UnfoldingOperMulti} of the unfolding operator for a
multi-valued function it follows that ${\cal T}_\eta(g_\eta)\in {\cal M}(\Omega\times Y,{\mathbb R}^N,p,\alpha,m)$. Therefore, due to this and convergence (\ref{PointwiseGraph}),
by Propositon~\ref {MainClassMaxMonoProp}(b) we obtain that
\begin{eqnarray}\label{ConvergeGraph}
{\cal T}_\eta({\cal G}_\eta)\rightarrowtail {\cal G}.
\end{eqnarray}
To prove that the limit functions $(\sigma_0,p_0)$ satisfy (\ref{HomogEqua4}), we apply 
Theorem~\ref{convMaxMonGraph}. Since the graph convergence is already established,
we show that condition (\ref{convMaxMonGraphCondition}) is fulfilled.
Using equations \eq{CurlPr1}
and \eq{CurlPr2}, we successfully compute that
\begin{eqnarray}\label{HelpLimitPass}
&&\frac1{|Y|}\int_{\Omega\times Y}({\cal T}_\eta(\partial_tp_\eta(t)),
{\cal T}_\eta(\Sigma^{\rm lin}_{\rm \eta}(t)))dxdy\non\\
&&=\frac1{|Y|}\int_{\Omega\times Y}\big(
{\cal T}_\eta\big(\partial_t(\sym(\na u_\eta(t)) - {\mathbb C}^{-1}\sigma_\eta(t)) \big),
{\cal T}_\eta(\sigma_\eta(t))\big)dxdy\non\\
&&+\frac1{|Y|}\int_{\Omega\times Y}\big(
{\cal T}_\eta(\partial_tp_\eta(t)),{\cal T}_\eta(
\Sigma^{\rm lin}_{\rm sh,\eta}(t)+\Sigma^{\rm lin}_{\rm curl,\eta}(t))\big)dxdy\non\\
&&=\int_{\Omega}(b(t),\partial_t u_\eta(t) )dx-
\frac1{|Y|}\int_{\Omega\times Y}\big(
{\cal T}_\eta(\partial_t{\mathbb C}^{-1}\sigma_\eta(t))),
{\cal T}_\eta(\sigma_\eta(t))\big)dxdy\non\\
&&-\frac1{|Y|}\int_{\Omega\times Y}\big(C_1
{\cal T}_\eta(\partial_t\dev\sym p_\eta(t)),
{\cal T}_\eta(\dev\sym p_\eta(t))\big)dxdy\non\\
&&-\frac1{|Y|}\int_{\Omega\times Y}\big(C_2
{\cal T}_\eta(\partial_t\Curl p_\eta(t)),
{\cal T}_\eta(\Curl p_\eta(t))\big)dxdy.\non
\end{eqnarray}
Integrating the last identity over $(0,t)$ and using the 
integration-by-parts formula we get that
\begin{eqnarray}\label{HelpLimitPass1}
&&\frac1{|Y|}\int_0^t({\cal T}_\eta(\partial_tp_\eta(t)),
{\cal T}_\eta(\Sigma^{\rm lin}_{\rm \eta}(t)))_{\Omega\times Y}dt=
\int_0^t(b(t),\partial_t u_\eta(t) )_\Omega dt\\
&&\hspace{6ex}-
\frac12\|{\cal T}_\eta({\cal B}^{1/2}\sigma_\eta(t))\|^2_{2,\Omega\times Y}
+\frac12\|{\cal T}_\eta({\cal B}^{1/2}\sigma_\eta(0))\|^2_{2,\Omega\times Y}\non\\
&&\hspace{6ex}-\frac12\|C_1^{1/2}{\cal T}_\eta(\dev\sym p_\eta(t))\|^2_{2,\Omega\times Y}
-\frac12\|C_2^{1/2}{\cal T}_\eta(\Curl p_\eta(t))\|^2_{2,\Omega\times Y},\non
\end{eqnarray}
where ${\cal B}={\mathbb C}^{-1}$. 
Moreover, since $\sigma_\eta(0)$ solves the linear elasticity problem \eq {elast1}  - 
\eq {elast3} with $\hat\ve_{\eta}=0$ and $\hat b=b(t)$,
by  \cite[Theorem 4.1]{Nesenenko12a}, we can conclude that
${\cal T}_\eta({\cal B}^{1/2}\sigma_\eta(0))$ converges to 
${\cal B}^{1/2}\sigma_0(0)$ strongly in $L^2(\Omega\times Y,{\cal S}^3)$.
Thus, by the lower semi-continuity of the norm the 
passing to the limit in \eq{HelpLimitPass1} yields
\begin{eqnarray}
&&\limsup_{n\to\infty}\frac1{|Y|}\int_0^t({\cal T}_\eta(\partial_tp_\eta(t)),
{\cal T}_\eta(\Sigma^{\rm lin}_{\rm \eta}(t)))_{\Omega\times Y}dt\non\\
&&\le\int_0^t(b(t),\partial_t u_0(t))_{\Omega}dt-
\frac12\|{\cal B}^{1/2}\sigma_0(t)\|^2_{2,\Omega\times Y}
+\frac12\|{\cal B}^{1/2}\sigma_0(0)\|^2_{2,\Omega\times Y}\non\\
&&\hspace{4ex}-\frac12\|C_1^{1/2}\dev\sym p_0(t)\|^2_{2,\Omega\times Y}
-\frac12\|C_2^{1/2}\Curl p_0(t)\|^2_{2,\Omega\times Y},\non
\end{eqnarray}
or
\begin{eqnarray}\label{HelpLimitPass2}
&&\limsup_{n\to\infty}\frac1{|Y|}\int_0^t({\cal T}_\eta(\partial_tp_\eta(t)),
{\cal T}_\eta(\Sigma^{\rm lin}_{\rm \eta}(t)))_{\Omega\times Y}dt\non\\
&&\le \int_0^t(b(t),\partial_t u_0(t))_{\Omega}dt-
\frac1{|Y|}\int_0^t\big(
\partial_t{\mathbb C}^{-1}\sigma_0(t),\sigma_0(t)\big)_{\Omega\times Y}dt\\
&&-\frac1{|Y|}\int_0^t\big(
\partial_t\dev\sym p_0(t), C_1\dev\sym p_0(t)\big)_{\Omega\times Y}dt\non\\
&&-\frac1{|Y|}\int_0^t\big(
\partial_t\Curl p_0(t), C_2\Curl p_0(t)\big)_{\Omega\times Y}dt\non
\end{eqnarray} 
We note that \eq{HelpEq1} and \eq{HelpEq2} imply
\begin{eqnarray}\label{HelpLimitPass3}
\int_{\Omega}(b(t),\partial_t u_0(t))dx =
\frac1{|Y|}\int_{\Omega\times Y}\big(\sigma_0(t),
\partial_t\ve(\na u_0(t)+\na_y u_1(t))\big)dxdy.
\end{eqnarray}
And, since for almost all $(x,y,t)\in\Omega\times Y\times(0,T_e)$ one has
\[\big(\partial_t\dev\sym p_0(x,y,t), C_1[y]\dev\sym p_0(x,y,t)\big)\] \[
=\big(\partial_tp_0(x,y,t), C_1[y]\dev\sym p_0(x,y,t)\big),\]
and that for almost all $t\in(0,T_e)$ 
\[\big(\partial_t\Curl p_0(t), C_2\Curl p_0(t)\big)_{\Omega\times Y}
=\big(\partial_tp_0(t), C_2\Curl\Curl p_0(t)\big)_{\Omega\times Y},\]
the relations \eq{HelpLimitPass2} and \eq{HelpLimitPass3} together with 
\eq{HelpEq0} yield
\begin{eqnarray}\label{HelpLimitPass4}
&& \limsup_{n\to\infty}\frac1{|Y|}\int_0^t({\cal T}_\eta(\partial_tp_\eta(t)),
{\cal T}_\eta(\Sigma^{\rm lin}_{\eta}(t)))_{\Omega\times Y}dt\non\\
&&\hspace{9ex}\le 
\frac1{|Y|}\int_0^t\big(
\partial_t p_0(t), \Sigma^{\rm lin}_0(t)\big)_{\Omega\times Y}dt.
\end{eqnarray}
In virtue of convergence (\ref{ConvergeGraph}) and
inequality (\ref{HelpLimitPass4}), Theorem~\ref{convMaxMonGraph} yields that
\[[\Sigma^{\rm lin}_0(x,y,t),\partial_t p_0(x,y,t)]\in Gr g(y)\] 
or, equivalently, that
\[\partial_t p_0(x,y,t)\in g(y,\Sigma^{\rm lin}_0(x,y,t)).\]
 The initial and boundary conditions (\ref{HomCurlPr4}) -
(\ref{HomCurlPr6}) for the limit functions $u_0, p$ and $p_0$ are easily obtained 
from the weak compactness of
$u_\eta$ and $p_\eta$ in the spaces $H^{1}(0,T_e; H^{1}_0(\Omega, {\mathbb R}^3))$
and $H^{1}(0,T_e;L^{2}(\Omega,{\cal M}^3))\cap 
L^{2}(0,{T_e};Z^{2}_{\Curl}(\Omega,{\cal M}^3))$, respectively. 
Therefore, summarizing everything done above, we conclude that the
functions $(u_0, u_1, \sigma, \sigma_0, p, p_0, p_1)$ 
satisfy the homogenized initial-boundary value
problem formed by the equations/inequalities \eq{HomogEqua1} - \eq{HomCurlPr6}.
\end{proof}

\hspace{-3ex}{\bf\large Acknowledgement.}
The author thanks the unknown reviewer for the critical reading of the manuscript 
and valuable comments which resulted in the improvements of the proofs and of the presentation 
of the obtained results. 

\bibliographystyle{plain} 
{\footnotesize
\bibliography{literatur1,literaturliste,literaturliste0}

\begin{thebibliography}{10}

\bibitem{AgmonDouglisNirenberg}
S.~Agmon, A.~Douglis, and L.~Nirenberg.
\newblock Estimates near the boundary for solutions of elliptic partial
  differential equations satisfying general boundary conditions. {II}.
\newblock {\em Comm. Pure Appl. Math.}, 17:35 -- 92, 1964.

\bibitem{Alb98}
H.-D. Alber.
\newblock {\em Materials with {M}emory. {I}nitial-{B}oundary {V}alue {P}roblems
  for {C}onstitutive {E}quations with {I}nternal {V}ariables}, volume 1682 of
  {\em Lecture Notes in Mathematics}.
\newblock Springer, Berlin, 1998.

\bibitem{Alb00}
H.-D. Alber.
\newblock Evolving microstructure and homogenization.
\newblock {\em Contin. Mech. Thermodyn.}, 12:235--286, 2000.

\bibitem{Alb03}
H.-D. Alber.
\newblock Justification of homogenized models for viscoplastic bodies with
  microstructure.
\newblock In K.~Hutter and H.~Baaser, editors, {\em Deformation and {F}ailure
  in {M}etallic {M}aterials.}, volume~10 of {\em Lecture Notes in Applied
  Mechanics}, pages 295--319. Springer, Berlin, 2003.

\bibitem{AlbNese09b}
H.-D. Alber and S.~Nesenenko.
\newblock Justification of homogenization in viscoplasticity: from convergence
  on two scales to an asymptotic solution in {$L^2(\Omega)$}.
\newblock {\em J. Multiscale Modeling}, 1(2):223--244, 2009.

\bibitem{Alla92}
G.~Allaire.
\newblock Homogenization and two-scale convergence.
\newblock {\em SIAM J. Math. Anal.}, 23(6):1482--1518, 1992.

\bibitem{Arbogast_Hornung_1990}
T.~Arbogast, J.~Douglas, and U.~Hornung.
\newblock Derivation of the double porosity model of single phase flow via
  homogenization theory.
\newblock {\em SIAM J. Math. Anal.}, 21:823--836, 1990.

\bibitem{Attouch84}
H.~Attouch.
\newblock {\em Variational {C}onvergence for {F}unctions and {O}perators}.
\newblock Applicable Mathematics Series. Pitman (Advanced Publishing Program),
  Boston, 1984.

\bibitem{Barb76}
V.~Barbu.
\newblock {\em Nonlinear {S}emigroups and {D}ifferential {E}quations in
  {B}anach {S}paces}.
\newblock Editura Academiei, Bucharest, 1976.

\bibitem{Bardella10}
L.~Bardella.
\newblock Size effects in phenomenological strain gradient plasticity
  constitutively involving the plastic spin.
\newblock {\em Internat. J. Engrg. Sci.}, 48(5):550--568, 2010.

\bibitem{Brez73}
H.~Br$\acute{e}$zis.
\newblock {\em Operateurs {M}aximaux {M}onotones}.
\newblock North Holland, Amsterdam, 1973.

\bibitem{Casado_Diaz_2000}
J.~Casado-Diaz.
\newblock Two-scale convergence for nonlinear dirichlet problems in perforated
  domains.
\newblock {\em Proc. Roy. Soc. Edinburgh Sec. A}, 130:249--276, 2000.

\bibitem{Casado_Diaz_2001}
J.~Casado-Diaz, M.~Luna-Laynez, and J.D. Martin-G$\acute{o}$men.
\newblock An adaptation of the multi-scale methods for the analysis of very
  thin reticulated structures.
\newblock {\em C. R. Acad. Sci. Paris, S{\'e}rie 1}, 332:223--228, 2001.

\bibitem{Castaing77}
C.~Castaing and M.~Valadier.
\newblock {\em Convex analysis and {M}easurable {M}ultifunctions}, volume 580
  of {\em Lecture Notes in Mathematics Studies}.
\newblock Springer, Berlin, 1977.

\bibitem{Cioranescu12}
D.~Cioranescu, A.~Damlamian, P.~Donato, G.~Griso, and R.~Zaki.
\newblock The periodic unfolding method in domains with holes.
\newblock {\em SIAM J. Math. Anal.}, 44(2):718--760, 2012.

\bibitem{Cioranescu02}
D.~Cioranescu, A.~Damlamian, and G.~Griso.
\newblock The periodic unfolding and homogenization.
\newblock {\em C. R. Acad. Sci. Paris Math}, 335(1):99--104, 2002.

\bibitem{Cioranescu08}
D.~Cioranescu, A.~Damlamian, and G.~Griso.
\newblock The periodic unfolding method in homogenization.
\newblock {\em SIAM J. Math. Anal.}, 40(4):1585--1620, 2008.

\bibitem{Damlamian07}
A.~Damlamian, N.~Meunier, and J.~Van Schaftingen.
\newblock Periodic homogenization of monotone multivalued operators.
\newblock {\em Nonlinear Anal., Theory Methods Appl.}, 67:3217--3239, 2007.

\bibitem{Damlamian09a}
A.~Damlamian, N.~Meunier, and J.~Van Schaftingen.
\newblock Periodic homogenization for convex functionals using mosco
  convergence.
\newblock {\em Ricerche mat.}, 57:209 -- 249, 2008.

\bibitem{Duvaut76}
G.~Duvaut and J.L. Lions.
\newblock {\em Inequalities in Mechanics and Physics.}
\newblock Springer, New-York, 1976.

\bibitem{Ebobisse_Neff09}
F.~Ebobisse and P.~Neff.
\newblock Rate-independent infinitesimal gradient plasticity with isotropic
  hardening and plastic spin.
\newblock {\em Math. Mech. Solids}, 15(6):691--703, 2010.

\bibitem{FleckWillis2004}
A.~Fleck and J.~R. Willis.
\newblock Bounds and estimates for the effect of strain gradients upon the
  effective plastic properties of an isotropic two-phase composite.
\newblock {\em J. Mech. Phys. Solids}, 52(8):1855 -- 1888, 2004.

\bibitem{Francfort2012}
G.~A. Francfort, A.Giacomini, and A.~Musesti.
\newblock On the {F}leck and {W}illis homogenization procedure in strain
  gradient plasticity.
\newblock {\em Preprint}, 2012.

\bibitem{Francu_2010}
J.~Francu.
\newblock Modification of unfolding approach to two-scale convergence.
\newblock {\em Mathematica Bohemica}, 135(4):402--412, 2010.

\bibitem{Francu_Svanstedt_2012}
J.~Francu and N.~E. Svanstedt.
\newblock Some remarks on two-scale convergence and periodic unfolding.
\newblock {\em Applications of Mathematics}, 57(4):359--375, 2012.

\bibitem{Lussardi08}
A.~Giacomini and L.~Lussardi.
\newblock Quasistatic evolution for a model in strain gradient plasticity.
\newblock {\em SIAM J. Math. Anal.}, 40(3):1201--1245, 2008.

\bibitem{Giacomini2011}
A.~Giacomini and A.~Musesti.
\newblock Two-scale homogenization for a model in strain gradient plasticity.
\newblock {\em ESAIM: Control, Optimisation and Calculus of Variations},
  17(4):1035--1065, 2011.

\bibitem{Gurtin05b}
M.E. Gurtin and L.~Anand.
\newblock A theory of strain-gradient plasticity for isotropic, plastically
  irrotational materials. {P}art {I}: {S}mall deformations.
\newblock {\em J. Mech. Phys. Solids}, 53:1624--1649, 2005.

\bibitem{Gurtin05}
M.E. Gurtin and A.~Needleman.
\newblock Boundary conditions in small-deformation, single crystal plasticity
  that account for the {B}urgers vector.
\newblock {\em J. Mech. Phys. Solids}, 53:1--31, 2005.

\bibitem{Han99}
W.~Han and B.D. Reddy.
\newblock {\em Plasticity. {M}athematical {T}heory and {N}umerical {A}nalysis.}
\newblock Springer, Berlin, 1999.

\bibitem{Hanke2011}
H.~Hanke.
\newblock Homogenization in gradient plasticity.
\newblock {\em Math. Models Methods Appl. Sci.}, 21(8):1651--1684, 2011.

\bibitem{Hu97}
Sh. Hu and N.~S. Papageorgiou.
\newblock {\em Handbook of {M}ultivalued {A}nalysis. {V}olume I: {T}heory}.
\newblock Mathematics and its {A}pplications. Kluwer, Dordrecht, 1997.

\bibitem{Kozono09a}
H.~Kozono and T.~Yanagisawa.
\newblock Global {D}iv-{C}url lemma on bounded domains in {${\mathbb R}^3$}.
\newblock {\em J. Funct. Anal.}, 256:3847 -- 3859, 2009.

\bibitem{Kozono09}
H.~Kozono and T.~Yanagisawa.
\newblock ${L}^r$-variational inequality for vector fields and the
  {H}elmholtz-{W}eyl decomposition in bounded domains.
\newblock {\em Indiana Univ. Math. J.}, 58(4):1853 -- 1920, 2009.

\bibitem{Kratochvil10}
J.~Kratochvil, M.~Kruzik, and R.~Sedlacek.
\newblock Energetic approach to gradient plasticity.
\newblock {\em Z. Angew. Math. Mech.}, 90(2):122--135, 2010.

\bibitem{Lenczner_1997}
M.~Lenczner.
\newblock Homog$\acute{e}$n$\acute{e}$isation d'un circuit
  $\acute{e}$lectrique.
\newblock {\em C. R. Acad. Sci. Paris II}, 324:537--542, 1997.

\bibitem{Mielke09}
A.~Mainik and A.~Mielke.
\newblock Global existence for rate-independent gradient plasticity at finite
  strain.
\newblock {\em J. Nonlinear Science}, 19(3):221--248, 2009.

\bibitem{Steinmann00}
A.~Menzel and P.~Steinmann.
\newblock On the continuum formulation of higher gradient plasticity for single
  and polycrystals.
\newblock {\em J. Mech. Phys. Solids}, 48:1777--1796, Erratum 49, (2001),
  1179--1180, 2000.

\bibitem{Miel07}
A.~Mielke and A.M. Timofte.
\newblock Two-scale homogenization for evolutionary variational inequalities
  via the energetic formulation.
\newblock {\em SIAM J. Math. Anal.}, 39(2):642--668, 2007.

\bibitem{Neff_techmech07}
P.~Neff.
\newblock Remarks on invariant modelling in finite strain gradient plasticity.
\newblock {\em Technische Mechanik}, 28(1):13--21, 2008.

\bibitem{Neff_Iutam08}
P.~Neff.
\newblock Uniqueness of strong solutions in infinitesimal perfect gradient
  plasticity with plastic spin.
\newblock In B.D. Reddy, editor, {\em {IUTAM}-{S}ymposium on {T}heoretical,
  {M}odelling and {C}omputational {A}spects of {I}nelastic {M}edia (in {C}ape
  {T}own, 2008)}, pages 129--140. Springer, Berlin, 2008.

\bibitem{Neff_Chelminski07_disloc}
P.~Neff, K.~Che{\l}mi\'nski, and H.D. Alber.
\newblock Notes on strain gradient plasticity. {F}inite strain covariant
  modelling and global existence in the infinitesimal rate-independent case.
\newblock {\em Math. Mod. Meth. Appl. Sci. (M3AS)}, 19(2):1--40, 2009.

\bibitem{Neff_Pauly_Witsch_cracad11}
P.~Neff, D.~Pauly, and K.J. Witsch.
\newblock A canonical extension of {K}orn's first inequality to
  {$H(\mathrm{Curl})$} motivated by gradient plasticity with plastic spin.
\newblock {\em C. R. Acad. Sci. Paris, Ser. I, doi:10.1016/j.crma.2011.10.003},
  2011.

\bibitem{Neff_Pauly_Witsch_Korn_diff_forms_m2as12}
P.~Neff, D.~Pauly, and K.J. Witsch.
\newblock Maxwell meets {K}orn: A new coercive inequality for tensor fields in
  {$\R^{N\times N}$} with square inegrable exterior derivative.
\newblock {\em Math. Meth. Appl. Sci.}, 35:65--71, 2012.

\bibitem{Neff_Pauly_Witsch_Sbornik12}
P.~Neff, D.~Pauly, and K.J. Witsch.
\newblock On a canonical extension of {K}orn's first and {P}oincar\'es
  inequality to {$H(\mathrm{Curl})$}.
\newblock {\em J. Math. Science (N.Y.)}, 2012.

\bibitem{Neff_Sydow_Wieners08}
P.~Neff, A.~Sydow, and C.~Wieners.
\newblock Numerical approximation of incremental infinitesimal gradient
  plasticity.
\newblock {\em Int. J. Num. Meth. Engrg.}, 77(3):414--436, 2009.

\bibitem{Nes07}
S.~Nesenenko.
\newblock Homogenisation in viscoplasticity.
\newblock {\em SIAM J. Math. Anal.}, 39(1):236--262, 2007.

\bibitem{Nesenenko12a}
S.~Nesenenko.
\newblock Homogenization of rate-dependent inelastic models of monotone type.
\newblock {\em Asymptot. Anal.}, 81:1--29, 2013.

\bibitem{NesenenkoNeff2011}
S.~Nesenenko and P.~Neff.
\newblock Well-posedness for dislocation based gradient visco-plasticity {I}:
  subdifferential case.
\newblock {\em SIAM J. Math. Anal.}, 44(3):1694--1712, 2012.

\bibitem{NesenenkoNeff2012}
S.~Nesenenko and P.~Neff.
\newblock Well-posedness for dislocation based gradient visco-plasticity {II}:
  general non-associative monotone plastic flows.
\newblock {\em Mathematics Mechanics of Complex Systems}, 1(2):149--176, 2013.

\bibitem{Nguetseng89}
G.~Nguetseng.
\newblock A general convergence result for a functional related to the theory
  of homogenization.
\newblock {\em SIAM J. Math. Anal.}, 20(3):608 -- 623, 1989.

\bibitem{Pankov97}
A.~Pankov.
\newblock {\em $G$-convergence and homogenization of nonlinear partial
  differential operators}.
\newblock Mathematics and its Applications. Kluwer, Dordrecht, 1997.

\bibitem{Pas78}
D.~Pascali and S.~Sburlan.
\newblock {\em Nonlinear {M}appings of {M}onotone {T}ype}.
\newblock Editura Academiei, Bucharest, 1978.

\bibitem{Phel93}
R.~R. Phelps.
\newblock {\em Convex {F}unctions, {M}onotone {O}perators and
  {D}ifferentiability}, volume 1364 of {\em Lecture Notes in Mathematics}.
\newblock Springer, Berlin, 1993.

\bibitem{Reddy06}
B.D. Reddy, F.~Ebobisse, and A.T. McBride.
\newblock Well-posedness of a model of strain gradient plasticity for
  plastically irrotational materials.
\newblock {\em Int. J. Plasticity}, 24:55--73, 2008.

\bibitem{Roubi05}
T.~Roubi\v{c}ek.
\newblock {\em Nonlinear {P}artial {D}ifferential {E}quations with
  {A}pplications}, volume 153 of {\em International Series of Numerical
  Mathematics}.
\newblock Birkh\"auser, Basel, 2005.

\bibitem{Schweizer10}
B.~Schweizer.
\newblock Homogenization of the {P}rager model in one-dimensional plasticity.
\newblock {\em Contin. Mech.Thermodyn.}, 20(8):459--477, 2009.

\bibitem{Sohr01}
H.~Sohr.
\newblock {\em The Navier-{S}tokes {E}quations: {A}n {E}lementary {F}unctional
  {A}nalytic {A}pproach}.
\newblock Advanced Texts. Birkh\"auser, Basel, 2001.

\bibitem{Neff_Svendsen08}
B.~Svendsen, P.~Neff, and A.~Menzel.
\newblock On constitutive and configurational aspects of models for gradient
  continua with microstructure.
\newblock {\em Z. Angew. Math. Mech. (special issue: Material Forces)},
  89(8):687--697, 2009.

\bibitem{Vale88}
T.~Valent.
\newblock {\em Boundary {V}alue {P}roblems of {F}inite {E}lasticity}.
\newblock Springer, Berlin, 1988.

\bibitem{Vis08b}
A.~Visintin.
\newblock Homogenization of nonlinear visco-elastic composites.
\newblock {\em J. Math. Pures Appl.}, 89(5):477--504, 2008.

\bibitem{Visintin08}
A.~Visintin.
\newblock Homogenization of the nonlinear {M}axwell model of viscoelasticity
  and of the {P}randtl-{R}euss model of elastoplasticity.
\newblock {\em Proc. Roy. Soc. Edinburgh Sec. A}, 138(6):1363--1401, 2008.

\end{thebibliography}
}

\end{document}